\pdfoutput=1
\documentclass[ps]{imsart}
\usepackage[utf8]{inputenc}
\usepackage[T1]{fontenc}

\usepackage{amsmath}
\usepackage{amssymb}
\usepackage{amsthm}
\usepackage{amsfonts}
\usepackage{graphicx}
\usepackage{wasysym}
\usepackage{amscd}
\usepackage{wrapfig}
\usepackage{stmaryrd}
\usepackage{color}
\usepackage{bbm}
\usepackage{verbatim}
\usepackage{multirow} %pour les tableaux
\usepackage{array}
\pubyear{2011}
\volume{0}
\issue{0}
\firstpage{1}
\lastpage{8}

\RequirePackage[colorlinks,citecolor=blue,urlcolor=blue]{hyperref}

\newtheorem{theorem}{Theorem}[section]
\newtheorem{conj}{Conjecture}
\newtheorem{corollary}[theorem]{Corollary}
\newtheorem{lemma}[theorem]{Lemma}
\newtheorem{proposition}[theorem]{Proposition}
\newtheorem{definition}[theorem]{Definition}

\newtheorem{remark}[theorem]{Remark}

\newtheorem{question}[conj]{Question}

\renewcommand{\P}{\mathbb P_{\!\frac12}}
\newcommand{\Pp}{\mathbb P_p}
\newcommand{\Z}{\mathbb{Z}}

\newcommand{\ep}{\varepsilon}

\newcommand{\CC}{{\mathbb C}}

\newcommand{\CLE}{{\mathrm{CLE}}}

\newcommand{\SLE}{{\mathrm{SLE}}}

\newcommand{\eg}{\emph{e.g.}}
\newcommand{\ie}{\emph{i.e.}}

\newcommand{\dd}{{\mathrm{d}}}
\newcommand{\decr}{\mathop{\makebox[0pt][l]{\kern0.5em$\downarrow$}\bigcap}}
\newcommand{\eqd}{:=}

\newcommand{\incr}{\mathop{\makebox[0pt][l]{\kern0.5em$\uparrow$}\bigcup}}

\endlocaldefs

\begin{document}

\begin{frontmatter}

\title{Planar percolation with a glimpse of Schramm--Loewner Evolution\thanksref{t1}}
\thankstext{t1}{This is an original survey paper}
\runtitle{Planar percolation with a glimpse of $\SLE$}

\author{\fnms{Vincent} \snm{Beffara}\ead[label=e1]{vbeffara@ens-lyon.fr}}
\address{Unité de Mathématiques Pures et Appliquées\\
         École Normale Supérieure de Lyon\\
         F-69364 Lyon CEDEX 7, France\\
         \printead{e1}}

\author{\fnms{Hugo} \snm{Duminil-Copin}\ead[label=e2]{hugo.duminil@unige.ch}}
\address{Département de Mathématiques\\
         Université de Genève\\
         Genève, Switzerland\\
         \printead{e2}}

\runauthor{Beffara and Duminil-Copin}

\vfill

\begin{center}
  \includegraphics[width=1.00\textwidth]{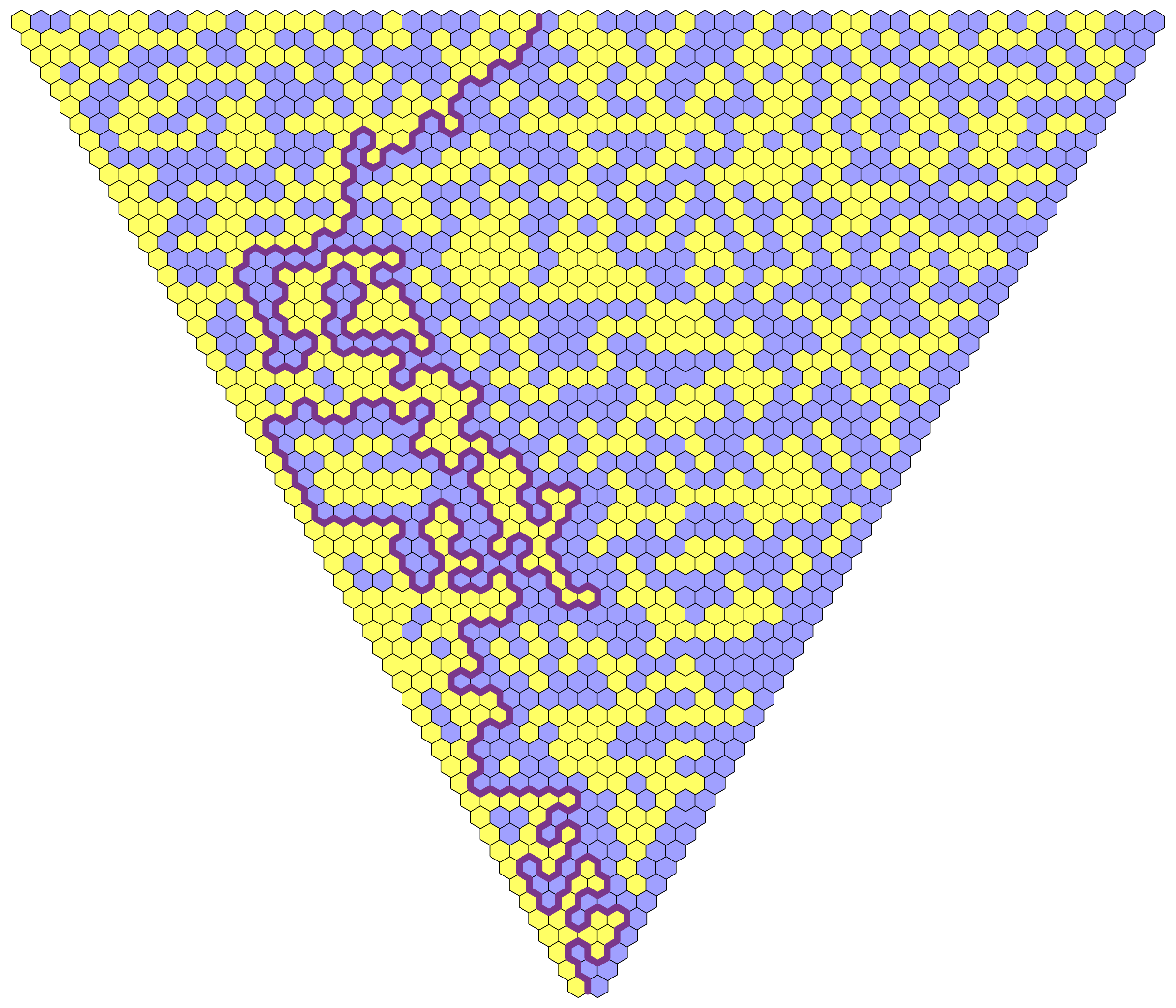}
\end{center}

\break

\begin{abstract}
  In recent years, important progress has been made in the field of
  two-dimensional statistical physics. One of the most striking
  achievements is the proof of the Cardy--Smirnov formula. This theorem,
  together with the introduction of Schramm--Loewner Evolution and
  techniques developed over the years in percolation, allow precise
  descriptions of the critical and near-critical regimes of the
  model. This survey aims to describe the different steps leading to the
  proof that the infinite-cluster density $\theta(p)$ for site
  percolation on the triangular lattice behaves like
  $(p-p_c)^{5/36+o(1)}$ as $p\searrow p_c=1/2$.
\end{abstract}

%\begin{keyword}[class=AMS]
%\kwd[Primary ]{}
%\kwd{}
%\kwd[; secondary ]{}
%\end{keyword}

\begin{keyword}
\kwd{Site percolation}
\kwd{Critical phenomenon}
\kwd{Conformal invariance}
\end{keyword}

% history:
% \received{\smonth{1} \syear{0000}}

\tableofcontents

\end{frontmatter}

\section{Introduction}

Percolation as a physical model was introduced by Broadbent and
Hammersley in the fifties~\cite{BH57}. For $p\in (0,1)$, \emph{(site)
  percolation} on the triangular lattice $\mathbb T$ is a random
configuration supported on the vertices (or \emph{sites}), each one
being \emph{open} with probability $p$ and \emph{closed} otherwise,
independently of the others. This can also be seen as a random
coloring of the faces of the hexagonal lattice $\mathbb H$ dual to $\mathbb T$.  Denote
the measure on configurations by $\mathbb P_p$. For general background
on percolation, we refer the reader to the books of
Grimmett~\cite{Gri99} and Kesten~\cite{Kes82}.

We will be interested in the connectivity properties of the model. Two
sets of sites $A$ and $B$ of the triangular lattice are
\emph{connected} (which will be denoted by $A\leftrightarrow B$) if
there exists an \emph{open path}, \ie\ a path of neighboring open
sites, starting at $a\in A$ and ending at $b\in B$. If there exists a
\emph{closed path}, \ie\ a path of neighboring closed sites, starting
at $a\in A$ and ending at $b\in B$, we will write
$A\stackrel{*}{\leftrightarrow}B$. If $A=\{a\}$ and $B=\{b\}$, we
simply write $a\leftrightarrow b$. We also write
$a\leftrightarrow\infty$ if $a$ is on an infinite open simple path.  A
\emph{cluster} is a connected component of open sites.

It is classical that there exists $p_c\in(0,1)$ such that for
$p<p_c$, there exists almost surely no infinite cluster,
while for $p>p_c$, there exists almost surely a unique such
cluster. This parameter is called the \emph{critical point}.

\begin{theorem}\label{thm:Kesten}
  The critical point of site-percolation on the triangular lattice
  equals $1/2$.
\end{theorem}

A similar theorem was first proved in the case of bond percolation on the
square lattice by Kesten in \cite{Kes80}.

%The proof can be summarized as follow. First, crossing probabilities
%at $p=1/2$ for rectangles $[0,n]\times[0,\rho n]$ are proved to remain
%bounded away from $0$ and $1$ uniformly in $n$. The main ingredients
%are self-duality at $p=1/2$ between open and closed sites and what is
%known as Russo--Seymour--Welsh theory. Second, crossing probabilities
%are proved to converge to $0$ and $1$ as $n$ goes to infinity when
%$p<1/2$ and $p>1/2$ respectively. The proof of this statement is based
%on a sharp threshold argument. From these two facts, one can conclude
%the proof by showing that connection probabilities decay exponentially
%fast below $1/2$, and that there is an infinite cluster above $1/2$
%almost surely.

Once the critical point has been determined, it is natural to study
the \emph{phase transition} of the model, \emph{i.e.} its behavior for
$p$ near $p_c$. Physicists are interested in the
thermodynamical properties of the model, such as the \emph{infinite
  cluster density}
$$\theta(p):=\Pp(0\leftrightarrow \infty)\text{ when
  $p>p_c$,}$$ the
\emph{susceptibility} (or \emph{mean
  cluster-size}) $$\chi(p):=\sum_{x\in \mathbb T}\Pp(0\leftrightarrow
x)\text{ when $p<p_c$,}$$ and the \emph{correlation length}
$L_p$ (see Definition~\ref{def:correlation_length}). The behavior of
these quantities near $p_c$ is believed to be governed by power laws:
\[\begin{array}{r@{\;}c@{\;}ll}
  \theta(p) &=& (p-p_c)^{\beta+o(1)} & \text{as }p\searrow p_c, \\
  \chi(p)    &=& (p-p_c)^{-\gamma+o(1)} & \text{as }p\nearrow p_c, \\
  L_p       &=& (p-p_c)^{-\nu+o(1)}  & \text{as }p\nearrow p_c. \\
\end{array}\]

These \emph{critical exponents} $\beta$, $\gamma$ and $\nu$ (and others) are
not independent of each other but satisfy certain equations called
\emph{scaling relations}. Kesten's scaling relations relate $\beta$, $\gamma$ and $\nu$
to the so-called monochromatic one-arm and polychromatic four-arm exponents at
criticality. The important feature of these relations is that they relate
quantities defined away from criticality to fractal properties of the critical
regime. In other words, the behavior of percolation through its phase
transition (as $p$ varies from slightly below to slightly above $p_c$) is
intimately related to its behavior at $p_c$. The scaling relations enable
mathematicians to focus on the critical phase. If the connectivity properties
of the critical phase can be understood, then critical exponents for $\theta$,
$\chi$, $L$ will follow.

We now turn to the study of planar percolation at $p=p_c$ and briefly
recall the history of the subject.  In the seminal
papers~\cite{BPZ84b} and~\cite{BPZ84a}, Belavin, Polyakov and
Zamolodchikov postulated \emph{conformal invariance} (under all
conformal transformations of sub-regions) in the scaling limit of
critical two-dimensional statistical mechanics models, of which
percolation at $p_c$ is one. The renormalization group formalism
suggests that the scaling limit of critical models is a fixed point
for the renormalization transformation. The fixed point being unique,
the scaling limit should be invariant under translation, rotation and
scaling. Since it can be described by quantum local fields, it is
natural to expect that the field describing the scaling limit of the
critical regime is itself invariant under all transformations which
are locally compositions of translations, rotations and
scalings. These transformations are exactly the conformal maps.

From a mathematical perspective, the notion of conformal invariance of
an entire model is ill-posed, since the meaning of scaling limit
depends on the object we wish to study (interfaces, size of clusters,
crossings, etc). A mathematical setting for studying scaling limits of
interfaces has been developed, therefore we will focus on this aspect
in this document.

Let us start with the study of a single curve. Fix a simply connected
planar domain $(\Omega,a,b)$ with two points on the boundary and
consider discretizations $(\Omega_\delta,a_\delta,b_\delta)$ of
$(\Omega,a,b)$ by a triangular lattice of mesh size $\delta$. The
clockwise boundary arc of $\Omega_\delta$ from $a_\delta$ to
$b_\delta$ is called $a_\delta b_\delta$, and the one from $b_\delta$
to $a_\delta$ is called $b_\delta a_\delta$. Assume now that the sites
of $a_\delta b_\delta$ are open and that those of $b_\delta a_\delta$
are closed. There exists a unique interface consisting of bonds of the
dual hexagonal lattice, between the open cluster of $a_\delta
b_\delta$ and the closed cluster of $b_\delta a_\delta$ (in order to
see this, the correspondence between face percolation on the hexagonal
lattice and site percolation on the triangular one is useful). We call
this interface the \emph{exploration path} and denote it by
$\gamma_\delta$; see the figure on the first page.

Conformal field theory leads to the prediction that $\gamma_\delta$ converges as
$\delta\rightarrow 0$ to a random, continuous, non-self-crossing curve
from $a$ to $b$ staying in $\Omega$, and which
is expected to be conformally invariant in the following sense.
\begin{definition}
  A family of random non-self-crossing continuous curves
  $\gamma_{(\Omega,a,b)}$, going from $a$ to $b$ and contained in $\Omega$,  indexed by simply connected
  domains with two marked points on the boundary $(\Omega,a,b)$ is
  \emph{conformally invariant} if for any $(\Omega,a,b)$ and any
  conformal map $\psi:\Omega\rightarrow \mathbb C$, $$\psi
  (\gamma_{(\Omega,a,b)})~\text{has the same law
    as}~\gamma_{(\psi(\Omega),\psi(a),\psi(b))}.$$
\end{definition}

In 1999, Schramm proposed a natural candidate for such conformally
invariant families of curves. He noticed that the interfaces of various models
satisfy the \emph{domain Markov property} (see
Section~\ref{sec:convergence_SLE}) which, together with the assumption
of conformal invariance, determines a one-parameter family of such
curves. In \cite{Sch00}, he introduced the Stochastic Loewner
evolution ($\SLE$ for short) which is now known as the
Schramm--Loewner evolution. For $\kappa>0$, a domain $\Omega$ and two
points $a$ and $b$ on its boundary, $\SLE(\kappa)$ is the random Loewner evolution in
$\Omega$ from $a$ to $b$ with driving process $\sqrt \kappa B_t$,
where $(B_t)$ is a standard Brownian motion. We refer to \cite{Wer09}
for a formal definition of $\SLE$. By construction, this process is
conformally invariant, random and fractal.  The prediction of
conformal field theory then translates into the following prediction
for percolation: the limit of $(\gamma_\delta)_{\delta>0}$ in
$(\Omega,a,b)$ is $\SLE(6)$.

For completeness, let us mention that when considering not only a
single curve but multiple interfaces, families of interfaces in a
percolation model are also expected to converge in the scaling limit
to a conformally invariant family of non-intersecting loops. Sheffield
and Werner \cite{SW10a, SW10b} introduced a one-parameter family of
probability measures on collections of non-intersecting loops which
are conformally invariant. These processes are called the Conformal
Loop Ensembles $\CLE(\kappa)$ for $\kappa>8/3$. The $\CLE(\kappa)$
process is related to the $\SLE(\kappa)$ in the following manner: the
loops of $\CLE(\kappa)$ are locally similar to $\SLE(\kappa)$.

Even though we now have a mathematical framework for conformal invariance, it
remains an extremely hard task to prove convergence of the interfaces in
$(\Omega_\delta,a_\delta,b_\delta)$ to $\SLE$.
Nevertheless, the observation that properties of interfaces should also be conformally invariant led
Langlands, Poulliot and Saint-Aubin to publish in \cite{LPSA94} numerical values
in agreement with the conformal invariance in the scaling limit of crossing
probabilities in the percolation model. More precisely, consider a Jordan
domain $\Omega$ with four points $A,B,C$ and $D$ on the boundary. The
$5$-tuple $(\Omega,A,B,C,D)$ is called a \emph{topological rectangle}. The
authors checked numerically that the probability $\mathcal C_\delta(\Omega,A,B,C,D)$ of
having a path of adjacent open sites between the boundary arcs $AB$ and $CD$
converges as $\delta$ goes to $0$ towards a limit which is the same for
$(\Omega,A,B,C,D)$ and $(\Omega',A',B',C',D')$ if they are images of each
other by a conformal map. Notice that the existence of such a crossing
property can be expressed in terms of properties of a well-chosen interface,
thus keeping this discussion in the frame proposed earlier.

The paper~\cite{LPSA94}, while only numerical, attracted many
mathematicians to the domain. The authors attribute the conjecture on
conformal invariance of the limit of crossing probabilities to
Aizenman. The same year, Cardy~\cite{Car92} proposed an explicit
formula for the limit. In 2001, Smirnov proved Cardy's formula
rigorously for critical site percolation on the triangular lattice,
hence rigorously providing a concrete example of a conformally
invariant property of the model.
\begin{theorem}[Smirnov \cite{Smi01}] \label{thm:cardy} For any topological rectangle $(\Omega,A,B,C,D)$, the
  probability of the event $\mathcal C_\delta(\Omega,A,B,C,D)$ has a
  limit $f(\Omega,A,B,C,D)$ as $\delta$ goes to $0$.  Furthermore, the
  limit satisfies the following two properties:
  \begin{itemize}
  \item It is equal to $AB/AC$ if $\Omega$ is an equilateral triangle
    with vertices $A$, $C$ and $D$;
  \item It is conformally invariant, in the following sense: if
    $\,\Phi$ is a conformal map from $\Omega$ to another simply
    connected domain $\Phi(\Omega)$, which extends continuously to
    $\partial\Omega$, then $$f(\Omega,A,B,C,D) = f(\Phi(\Omega),
    \Phi(A), \Phi(B) ,\Phi(C) ,\Phi(D)).$$
  \end{itemize}
\end{theorem}

The fact that Cardy's formula takes such a simple form for equilateral
triangles was first observed by Carleson. Notice that the Riemann
mapping theorem along with the second property give the value of $f$
for every conformal rectangle.

\medbreak

A remarkable consequence of this theorem is that, even though Cardy's
formula provides information on crossing probabilities only, it can in
fact be used to prove much more. We will see that it
implies convergence of interfaces to the \emph{trace} of $\SLE(6)$
(see Section~\ref{sec:convergence_SLE}). In other words, conformal
invariance of one well-chosen quantity can be sufficient to prove
conformal invariance of interfaces.

\begin{theorem}[Smirnov, see also \cite{CN07}] \label{thm:SLE} Let
  $\Omega$ be a simply connected domain with two marked points $a$ and
  $b$ on the boundary. Let $\gamma_\delta$ be the exploration path of
  critical percolation as described in the previous paragraphs. Then
  the law of $\gamma_\delta$ converges weakly, as $\delta\rightarrow
  0$, to the law of the trace of $\SLE(6)$ in $(\Omega,a,b)$.
\end{theorem}

Here and in later statements, the topology is associated to the distance on
curves in $\Omega$ from $a$ to $b$ defined by $$d(\Gamma,\tilde
\Gamma)=\inf_{\varphi}\sup_{t\ge0}|\Gamma(t)-\tilde \Gamma(\varphi(t))|,$$
where the infimum is over all strictly increasing functions from $\mathbb R_+$
onto itself.

Similarly, one can consider the convergence of the whole family of
discrete interfaces between open and closed clusters. This family
converges to $\CLE(6)$, as was proved in \cite{CN06}, thus providing a
proof of the full conformal invariance of percolation interfaces.

\bigskip

Convergence to $\SLE(6)$ is important for many reasons. Since $\SLE$
itself is very well understood (its fractal properties in particular),
it enables the computation of several critical exponents describing
the critical phase. We will introduce these exponents later in the
survey. For now we state the result informally (see
Theorem~\ref{exponents} or \cite{SW01}):
\begin{itemize}
\item the probability that there exists an open path from the origin
  to the boundary of the box of radius $n$ behaves as $n^{-5/48+o(1)}$ as
  $n$ tends to infinity;
\item the probability that there exist four arms, two open and two
  closed, from the origin to the boundary of the box of size $n$,
  behaves as $n^{-5/4+o(1)}$ as $n$ tends to infinity.
\end{itemize}

Together with Kesten's scaling relations (Theorem~\ref{CD} or
\cite{Kes87}), the previous asymptotics imply the following result,
which is the main focus of this survey:

\begin{theorem}\label{main_theorem}
  For site percolation on the triangular lattice, $p_c=1/2$
  and $$\theta(p) = (p-1/2)^{5/36+o(1)}\qquad\text{as $p\searrow
    1/2$}.$$
\end{theorem}

\begin{figure}[ht!]
  \begin{center}
    \includegraphics[width=0.45\textwidth]{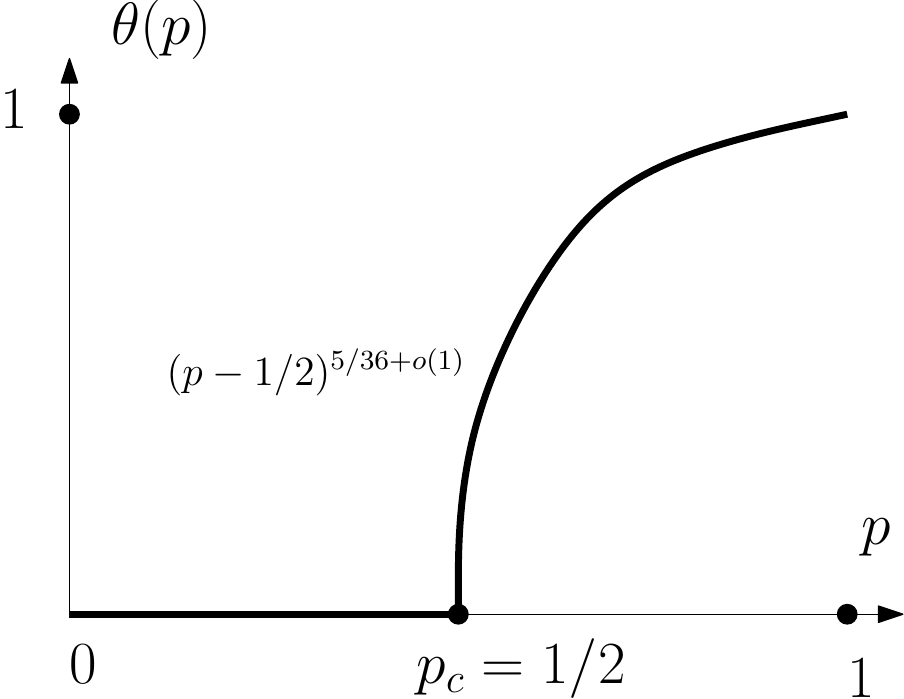}
  \end{center}
  \caption{\label{fig:theta}Cluster density with respect to
    $p$. Non-trivial facts in this picture include $p_c=1/2$,
    $\theta(p_c)=0$ and the behavior of $p\rightarrow
    \theta(p)$ near the critical point.}
\end{figure}

\subsection*{Organization of the survey}

The next section is devoted to the geometry of percolation with
$p=1/2$. First, we obtain uniform bounds for box crossing
probabilities (via a RSW-type argument). Then, we prove the Cardy-Smirnov
formula (Theorem~\ref{thm:cardy}). Finally, we sketch the proof of
convergence to $\SLE(6)$ (Theorem~\ref{thm:SLE}).

The second section deals with critical exponents at criticality. We present the
derivation of \emph{arm-exponents} assuming some basic estimates on
$\SLE$ processes.

The third section studies percolation away from $p=1/2$. We prove that
$p_c=\frac12$ and we introduce the notion of correlation
length for general $p$. Then, we study the properties of percolation
at scales smaller than the correlation length. Finally, we investigate
Kesten's scaling relations and prove Theorem~\ref{main_theorem}.

The last section gathers a few open questions relevant to the topic.

\subsection*{Notation and standard correlation inequalities in percolation}

\paragraph{Lattice, distance and balls} Except if otherwise stated,
$\mathbb T$ will denote the triangular lattice with mesh size $1$,
embedded in the complex plane $\CC$, containing a vertex at the origin
and a vertex at $1$. Complex coordinates will be used frequently to
specify the location of a point.  Let $d_{\mathbb T}(\cdot,\cdot)$ be
the graph distance in $\mathbb T$. Define the ball $\Lambda_n:=\{x\in
\mathbb T:d_{\mathbb T}(x,0)\le n\}$ (balls have hexagonal
shapes). Let $\partial \Lambda_n = \Lambda_n\setminus \Lambda_{n-1}$
be the internal boundary of $\Lambda_n$.

\paragraph{Increasing events} The Harris inequality and the monotonicity of percolation will be used
a few times. We recall these two facts now. An event is called
\emph{increasing} if it is preserved by the addition of open sites,
see Section~2.2 of \cite{Gri99} (a typical example is the existence of
an open path from one set to another). The inequality $p<p'$ implies
that $\Pp(A)\le \mathbb P_{p'}(A)$ for any increasing event
$A$. Moreover, for every $p\in[0,1]$ and $A$, $B$ two increasing
events,
$$\Pp(A\cap B)\ge\Pp(A)\Pp(B)\quad\quad\text{(Harris inequality)}.$$
The Harris inequality is a particular case of the
Fortuin-Kasteleyn-Ginibre inequality \cite{FKG71}.

The van den Berg-Kesten inequality \cite{vdBK85} will also be used
extensively. %For a set of sites $F$ and a configuration $\omega$, let $\omega_F$ be the configuration constructed from $\omega$ by closing all sites outside of $F$.
For two increasing events $A$ and $B$, let $A\circ B$ be the event
that $A$ and $B$ occur disjointly, meaning that $\omega\in A\circ B$
if and only if there exist a set of sites $E$ (possibly depending on
$\omega$) such that any configuration $\omega'$ with
\smash{$\omega'_{|E}=\omega_{|E}$} is in $A$ and any configuration
$\omega''$ with $\omega''_{|\mathbb T\setminus E}=\omega_{|\mathbb
  T\setminus E}$ is in $B$. In words, the state of sites in $E$ is
sufficient to verify whether $\omega$ is in $A$ or not, and similarly
for $\mathbb T\setminus E$ for $B$. For instance, the event
$\{a\leftrightarrow b\}\circ\{c\leftrightarrow d\}$, for $a,b,c,d$
four disjoint sites is the event that there exist two disjoint paths
connecting $a$ to $b$ and $c$ to $d$ respectively. It is different
from the event $\{a\leftrightarrow b\}\cap\{c\leftrightarrow d\}$
which requires only that there exist two paths connecting $a$ to $b$
and $c$ to $d$, but not necessarily disjoint.

For every $p\in[0,1]$ and $A$, $B$ two increasing events depending {\em on a finite number of sites},
$$\Pp(A\circ B)\le \Pp(A)\Pp(B)\quad\quad\text{(BK inequality)}.$$
This inequality was improved by Reimer \cite{Rei00}, who proved that
the inequality is true for any two (non-necessarily increasing) events
$A$ and $B$ depending on a finite number of sites.

\subsection*{References}

For general background on percolation, we refer the reader to the
books of Grimmett~\cite{Gri99}, Bollob\'as and Riordan \cite{BR} and Kesten~\cite{Kes82}. The proof of
Cardy's formula can be found in the original paper
\cite{Smi01}. Convergence of interfaces to $\SLE$ is proved in
\cite{CN07}. Scaling relations can be found in \cite{Kes87,Nol08}. Lawler's
book \cite{Law05} and Sun's review \cite{Sun11} are good places to get
a general account on $\SLE$. We also refer to original research
articles on the subject. More generally, subjects treated in this
review are very close to those studied in Werner's lecture notes
\cite{Wer09}.

\section{Crossing probabilities and conformal invariance at the
  critical point}

\subsection{Circuits in annuli}

In this whole section, we let $p=1/2$. Let $\mathcal E_n$ be the event
that there exists a circuit (meaning a sequence of neighboring sites $x_1,\dots,x_n,x_1$) of open sites in
$\Lambda_{3n}\setminus \Lambda_n$ that surrounds the origin.
\begin{theorem}\label{RSW}
  There exists $C>0$ such that for every $n>0$, $\P(\mathcal E_n) \ge
  C$.
\end{theorem}

This theorem was first proved in a corresponding form in the case of
bond percolation on the square lattice by Russo \cite{Rus78} and by
Seymour and Welsh \cite{SW78}. It has many applications, several of
which will be discussed in this survey.

Such a bound (and its proof) is not surprising since open and closed
sites play symmetric roles at $p=\frac12$. It is natural to
expect that the probability of $\mathcal E_n$ goes to $0$
(resp.\ $1$) for $p$ below (resp.\ above) $1/2$.

\begin{proof}
  We present one of the many proofs of Theorem~\ref{RSW}, inspired by
  an argument due to Smirnov and presented in \cite{Wer09f} (in
  French).

  \paragraph{Step 1:} Let $n>0$ and index the sides of $\Lambda_n$ as
  in Fig.~\ref{fig:RSWstep2}. Consider the event that $\ell_1$ is
  connected by an open path to $\ell_3\cup \ell_4$ in $\Lambda_n$. The
  triangular lattice being a triangulation, the complement of this
  event is that $\ell_2$ is connected by a closed path to $\ell_5\cup
  \ell_6$ in $\Lambda_n$. Using the symmetry between closed and open
  sites and the invariance of the model under rotations of angle
  $\pi/3$ around the origin, \smash{$\P(\ell_1\leftrightarrow
    \ell_3\cup\ell_4\text{ in }\Lambda_n)$} is equal to $1/2$. Let us
  emphasize that we used that $\mathbb T$ is a triangulation invariant
  under rotations of angle $\pi/3$.

  In fact, we also have that $\smash{\P}(\ell_1\leftrightarrow
  \ell_4\text{ in }\Lambda_n)\ge 1/8$. Indeed, either this is true or,
  going to the complement, $\P(\ell_1\leftrightarrow \ell_3\text{ in
  }\Lambda_n)\ge 1/2-1/8$. But in this case, using the Harris
  inequality, $$\P(\ell_1\leftrightarrow \ell_4\text{ in
  }\Lambda_n)\ge\P(\ell_1\leftrightarrow \ell_3\text{ in
  }\Lambda_n)\P(\ell_2\leftrightarrow \ell_4\text{ in }\Lambda_n)\ge
  (3/8)^2\ge1/8.$$

  \paragraph{Step 2:} Let $i=\sqrt{-1}$. Consider $R_n=\Lambda_n\cup
  (\Lambda_n-\sqrt3n{i})$ and index the sides of $R_n$ as in
  Fig.~\ref{fig:RSWstep2}. For a path $\gamma$ from $\ell_1$ to
  $\ell_4$ in $\Lambda_n$, define the domain $\Omega_\gamma$ to
  consist of the sites of $R_n$ strictly to the right of $\gamma\cup
  \sigma(\gamma)$, where $\sigma$ is the reflection with respect to
  $\ell_1$. Once again, the complement of $\{\ell_4\cup
  \gamma\leftrightarrow\ell_{10}\cup \ell_{11}$ in $\Omega_\gamma\}$
  is \smash{$\{\ell_9\cup\sigma(\gamma)\stackrel{*}{\leftrightarrow}
    \ell_2\cup\ell_3$ in $\Omega_\gamma\}$}. The switching of colors
  and the symmetry with respect to $\ell_1$ imply that the probability
  of the former is at least $1/2$ (it is not equal to $1/2$ since the
  site on $\ell_1$ is necessarily open).

  \begin{figure}[ht!]
    \begin{center}
      \includegraphics[width=0.40\textwidth]{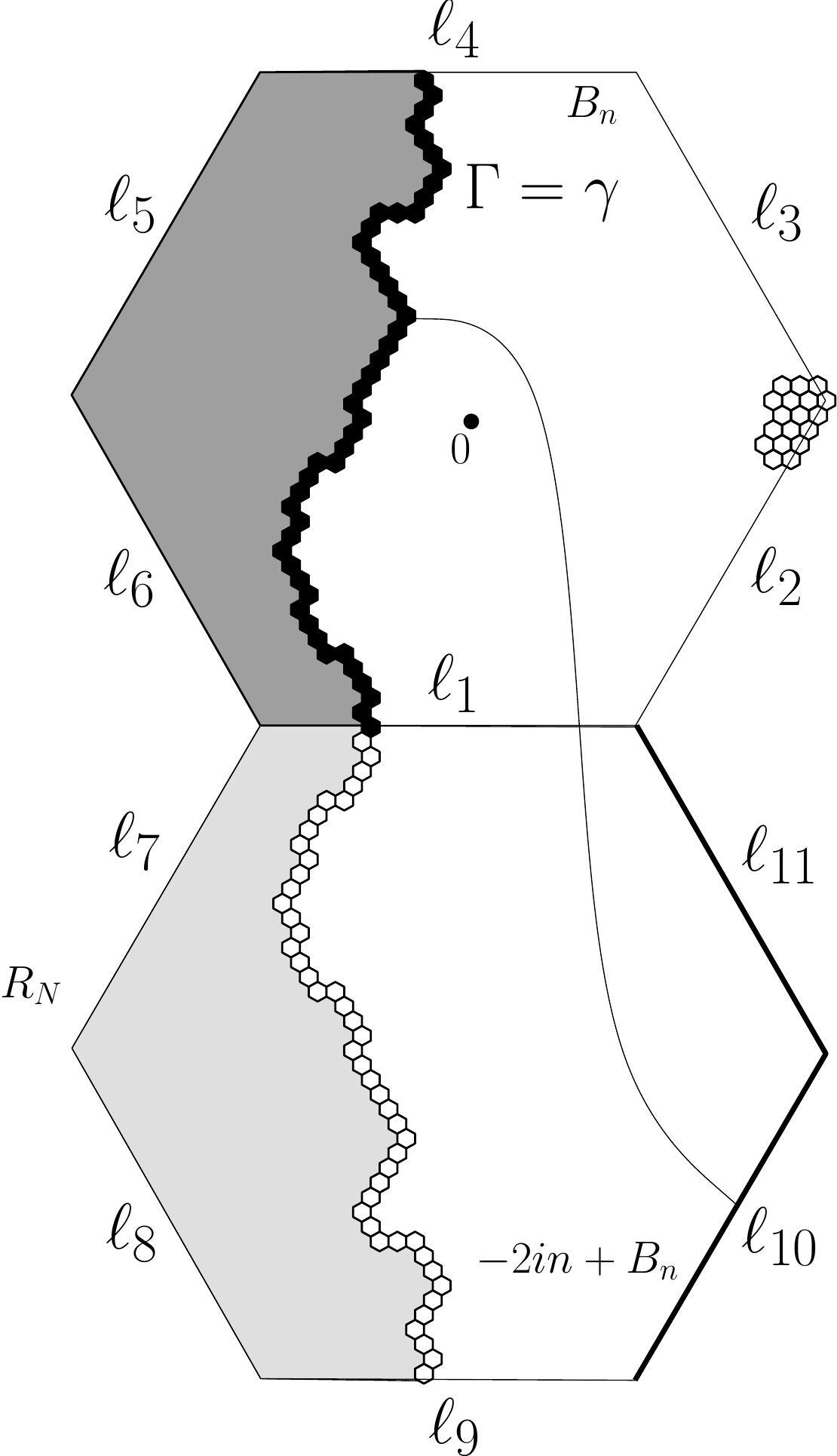}
    \end{center}
    \caption{\label{fig:RSWstep2} The dark gray area is the set of
      sites which are discovered after conditioning on
      $\{\Gamma=\gamma\}$. The white area is $\Omega_\gamma$.}
  \end{figure}

  If $E:=\{\ell_1\leftrightarrow \ell_4\text{ in $\Lambda_n$}\}$ occurs, set
  $\Gamma$ to be the left-most crossing between $\ell_1$ and
  $\ell_4$. For a given path $\gamma$ from $\ell_1$ to $\ell_4$, the
  event $\{\Gamma=\gamma\}$ is measurable only in terms of sites to
  the left or in $\gamma$. In particular, conditioning on
  $\{\Gamma=\gamma\}$, the configuration in $\Omega_\gamma$ is a
  percolation configuration. Thus,
  $$\P\big((\ell_4\cup \gamma)\leftrightarrow (\ell_{10}\cup
  \ell_{11})\text{ in }\Omega_\gamma~|~\Gamma=\gamma\big)~\ge~1/2.$$
  Therefore,
  \begin{align*}
    \P\big(\ell_4\leftrightarrow (\ell_{10}\cup \ell_{11})\text{~in
      $R_n$}\big)~&=~\P\big(\ell_4\leftrightarrow (\ell_{10}\cup
    \ell_{11})\text{~in $R_n$}~,~E\big)\\
    &=~\sum_{\gamma}\P\big(\ell_4\leftrightarrow (\ell_{10}\cup
    \ell_{11})\text{~in $R_n$}~,~\Gamma=\gamma\big)\\
    &\ge~\sum_{\gamma}\P\big((\ell_4\cup \gamma)\leftrightarrow
    (\ell_{10}\cup
    \ell_{11})\text{~in~}\Omega_\gamma~,~\Gamma=\gamma\big)\\
    &\ge~\sum_{\gamma}\frac12\P(\Gamma=\gamma)~=~\frac
    12\P(E)~\ge~\frac1{16}.
  \end{align*}

  \begin{figure}[ht!]
    \begin{center}
      \includegraphics[width=0.40\textwidth]{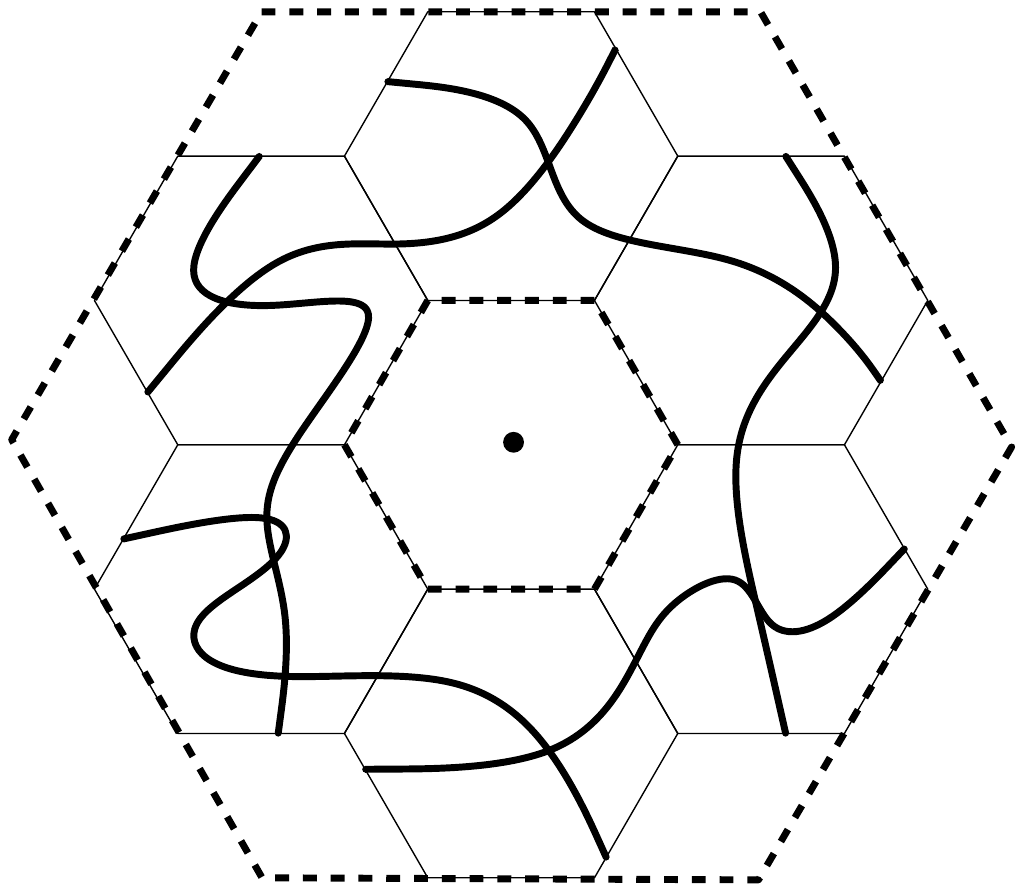}
    \end{center}
    \caption{\label{fig:RSWstep3} Six ``rectangles'' which, when
      crossed, ensure the existence of a circuit in the annulus.}
  \end{figure}

  \paragraph{Step 3:} Invoking the Harris inequality,
  $$\P(\ell_4\leftrightarrow \ell_{9})~\ge~\P\big(\ell_4\leftrightarrow
  (\ell_{10}\cup \ell_{11})\big)\P\big((\ell_2\cup
  \ell_3)\leftrightarrow \ell_{9}\big)~\ge\frac1{16^2}.$$ Assuming
  that the six subdomains of the space (which correspond to translations and rotations of $R_n$) described in Fig.~\ref{fig:RSWstep3} are
  crossed (in the sense that there are open paths between opposite
  short edges), the result follows from a final use of the Harris
  inequality.
\end{proof}

The first corollary of Theorem~\ref{RSW} is the following lower bound
on $p_c$. The result can also be proved without Theorem~\ref{RSW} using an elegant
argument by Zhang which invokes the uniqueness of the infinite cluster
when it exists (see Section~11 of \cite{Gri99}).

\begin{corollary}[Harris \cite{Har60}] \label{lower_bound_critical}
  For site percolation on the triangular lattice, $\theta(1/2)=0$. In
  particular, $p_c\ge 1/2$.
\end{corollary}

\begin{proof}
  Let us prove that when $p=1/2$, $0$ is almost surely not connected
  by a closed path to infinity (it is the same probability for an open
  path). Let $N>0$. We consider the $N$ concentric disjoint annuli
  $\Lambda_{3^{n+1}}\setminus\Lambda_{3^n}$, for $0\le n<N$, and we
  use that the behavior in each annulus is independent of the behavior
  in the others. Formally, the origin being connected to $\partial
  \Lambda_{3^N}$ by a closed path implies that for every $n< N$, the
  complement, $\mathcal E^c_{3^n}$, of $\mathcal E_{3^n}$
  occurs. Therefore,
  \begin{align}\label{upper_expo}
    \P(0\stackrel{*}{\leftrightarrow} \partial
    \Lambda_{3^N})~&\le~\P\left(\bigcap_{n<N}\mathcal
      E^c_{3^n}\right)~=~\prod_{n<N}\P\left(\mathcal
      E^c_{3^n}\right)~\le~(1-C)^N,
  \end{align}
  where $C$ is the constant in Theorem~\ref{RSW}.  In the second
  inequality, the independence between percolation in different annuli
  is crucial. In particular, the left-hand term converges to $0$ as
  $N\rightarrow \infty$, so that $\theta(1/2)=0$. Hence, by the
  definition of $p_c$, $p_c\ge 1/2$.
\end{proof}

\subsection{Discretization of domains and crossing probabilities}

In general, we are interested in crossing probabilities for general
shapes. Consider a topological rectangle $(\Omega,A,B,C,D)$,
\emph{i.e.} a simply connected domain $\Omega\neq \mathbb C$ delimited
by a non-intersecting continuous curve and four distinct points $A$,
$B$, $C$ and $D$ on its boundary, indexed in counter-clockwise
order. The eager reader might want to check that the argument of this
section still goes through without the assumption that the boundary is
a simple curve, when $A$, $B$, $C$ and $D$ are prime ends of $\Omega$
--- in fact, this extension is needed if one wants to prove
convergence to $\SLE_6$, because the boundary of a stopped $\SLE$ will
typically \emph{not} be a simple curve.

For $\delta>0$, we will be interested in percolation on
$\Omega_\delta:=\Omega\cap\delta\mathbb T$ given by vertices of
$\delta\mathbb T$ in $\Omega$ and edges entirely included in
$\Omega$. Note that the boundary of $\Omega_\delta$ can be seen as a
self-avoiding curve $s$ on $\Omega_\delta^*$ (which is a subgraph of
the hexagonal lattice). Once again, this may not be true if the domain is not smooth, but we choose not to discuss this matter here. The graph $\Omega_\delta$ should be seen as a
discretization of $\Omega$ at scale $\delta$.  Let $A_\delta$,
$B_\delta$, $C_\delta$ and $D_\delta$ be the dual sites in $s$ that
are closest to $A$, $B$, $C$ and $D$ respectively. They divide $s$
into four arcs denoted by $A_\delta B_\delta$, $B_\delta C_\delta$,
etc.

In the percolation setting, let $\mathcal C_\delta(\Omega,A,B,C,D)$ be
the event that there is a path of open sites in $\Omega_\delta$
between the intervals $A_\delta B_\delta$ and $C_\delta D_\delta$ of
its boundary (more precisely connecting two sites of $\Omega_\delta$
adjacent to $A_\delta B_\delta$ and $C_\delta D_\delta$
respectively). We call such a path an \emph{open crossing}, and the
event a \emph{crossing event}; accordingly we will say that the
rectangle is \emph{crossed} if there exists an open crossing.

With a slight abuse of notation, we will denote the percolation
measure with $p=1/2$ on $\delta\mathbb T$ by $\P$ (even though the
measure is the push-forward of $\P$ by the scaling $x\mapsto \delta
x$). We first state a direct consequence of Theorem~\ref{RSW}:

\begin{corollary}[Rough bounds on crossing probabilities]\label{any_shape}
  Let $(\Omega,A,B,C,D)$ be a topological rectangle. There exist
  $0<c_1,c_2<1$ such that for every $\delta>0$,
  \begin{equation*}
    c_1~\le~\P\big[\mathcal C_\delta(\Omega,A,B,C,D)\big]~\le~c_2.
  \end{equation*}
\end{corollary}

\begin{figure}[ht!]
  \begin{center}
    \includegraphics[width=0.50\textwidth]{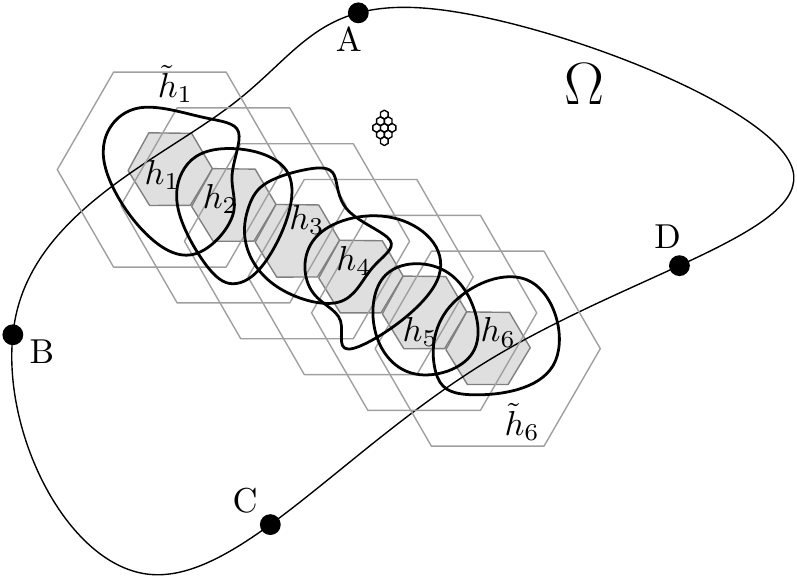}
  \end{center}
  \caption{\label{fig:domain_omega} Circuits in annuli linking two
    arcs of a topological rectangle. If each of these annuli contains
    an open circuit disconnecting the interior from the exterior
    boundary, we obtain an open path connecting the two sides.}
\end{figure}

\begin{proof}
  It is sufficient to prove the lower bound, since the upper bound is
  a consequence of the following fact: the complement of $\mathcal
  C_\delta(\Omega,A,B,C,D)$ is the existence of a closed path from
  $B_\delta C_\delta$ to $D_\delta A_\delta$, it has same probability
  as $\mathcal C_\delta(\Omega,B,C,D,A)$. Therefore, if the latter
  probability is bounded from below, the probability of $\mathcal
  C_\delta(\Omega,A,B,C,D)$ is bounded away from $1$.

  Fix $\varepsilon \in \delta \mathbb N$ positive. For a hexagon $h$
  of radius $\varepsilon>0$, we set $\tilde h$ to be the hexagon with
  the same center and radius $3\varepsilon$. Now, consider a
  collection $h_1$, \ldots, $h_k$ of hexagons ``parallel'' to the
  hexagonal lattice $\mathbb H$ (the dual lattice of $\mathbb T$) and
  of radius $\varepsilon$ satisfying the following conditions:
  \begin{itemize}
  \item $h_1$ intersects $AB$ and $h_k$ intersects $CD$,
  \item $\tilde h_1$, \ldots, $\tilde h_k$ intersect neither $BC$ nor
    $DA$,
  \item $h_i$ are adjacent and the union of hexagons $h_i$ connects
    $AB$ to $CD$ in $\Omega$.
  \end{itemize}
  For any domain and any $\delta>0$ small enough, $\varepsilon>0$ can
  be chosen small enough so that the family $(h_i)$ exists.

  Let $E_i^\delta$ be the event that there is an open circuit in
  $\Omega_\delta\cap(\tilde h_i\setminus h_i)$ surrounding
  $\Omega_\delta\cap h_i$. By construction, if each $E_i^\delta$
  occurs, there is a path from $AB$ to $CD$, see
  Fig.~\ref{fig:domain_omega}. Using Theorem~\ref{RSW}, the
  probability of this is bounded from below by $C^k$, where $C$ does
  not depend on $\varepsilon$ and $\delta$. Now, there exists a constant $K=K(\Omega)$ such that there is a choice of $\ep>0$, $h_1,\dots,h_k$, with $k\le K$ working for any $\delta$ small enough, a fact which implies the
  claim.
\end{proof}

In particular, long rectangles are crossed in the long direction with
probability bounded away from $0$ as $\delta\rightarrow 0$. This
result is the classical formulation of Theorem~\ref{RSW}. We finish
this section with a property of percolation with parameter $1/2$:

\begin{corollary}\label{critical_exponent}
  There exist $\alpha,\beta>0$ such that for every $n>0$,
  $$n^{-\alpha}~\le~\P(0\leftrightarrow \partial
  \Lambda_n)~\le~n^{-\beta}.$$
\end{corollary}

\begin{proof}
  The existence of $\beta>0$ is proved as in \eqref{upper_expo}. For
  the lower bound, we use the following construction. Define

  $$R_n:=\big\{k\cdot 1+\ell\cdot e^{i\pi/3}:k\in[0,2^n]\text{ and }
  \ell\in[0,2^{n+1}]\big\}$$ if $n$ is odd, and
  $$R_n:=\big\{k\cdot 1+\ell\cdot e^{i\pi/3}:k\in[0,2^{n+1}] \text{
    and }\ell\in[0,2^n]\big\}$$ if it is even. Set $F_n$ to be the
  event that $R_n$ is crossed in the ``long''
  direction. Corollary~\ref{any_shape} implies the existence of
  $C_1>0$ such that $\P(F_n)\ge C_1$ for every $n>0$. By the Harris
  inequality $$\P(0\leftrightarrow \partial
  \Lambda_{3^N})~\ge\P\left(\bigcap_{n=0}^NF_n\right)~\ge~\prod_{n=0}^N
  \P(F_n)~\ge~C_1^{N+1}.$$ This yields the existence of $\alpha>0$.
\end{proof}

\subsection{The Cardy--Smirnov formula}
\label{sec:Cardy_s_formula}

The subject of this section is the proof of
Theorem~\ref{thm:cardy}. The proof of this theorem is very well (and
very shortly) exposed in the original paper \cite{Smi01}. It has been
rewritten in a number of places including \cite{BR06c,Gri10,Wer09}. We
provide here a version of the proof which is mainly inspired by
\cite{Smi01} and \cite{Bef07}.

\begin{proof}
  Fix $(\Omega,A,B,C)$ a topological triangle and $z\in \Omega$ (with
  the same caveat as in the previous proof, we will silently assume
  the boundary of $\Omega$ to be smooth and simple, for notation's sake, but the
  same argument applies to the general case of a simply connected
  domain). For $\delta>0$, $A_\delta$, $B_\delta$, $C_\delta$, $z_\delta$ are
  the closest points of $\Omega_\delta^*$ to $A$, $B$, $C$, $z$ respectively, as
  before. Define $E_{A,\delta}(z)$ to be the event that there exists
  a non-self-intersecting path of open sites in $\Omega_\delta$,
  separating $A_\delta$ and $z_\delta$ from $B_\delta$ and
  $C_\delta$. We define $E_{B,\delta}(z)$, $E_{C,\delta}(z)$
  similarly, with obvious circular permutations of the letters.  Let
  $H_{A,\delta}(z)$ (resp.\ $H_{B,\delta}(z)$, $H_{C,\delta}(z)$) be
  the probability of $E_{A,\delta}(z)$ (resp.\ $E_{B,\delta}(z)$,
  $E_{C,\delta}(z)$). The functions $H_{A,\delta}$, $H_{B,\delta}$ and
  $H_{C,\delta}$ are extended to piecewise linear functions on
  $\Omega$.

  \begin{figure}[ht!]
    \begin{center}
      \includegraphics[width=0.50\textwidth]{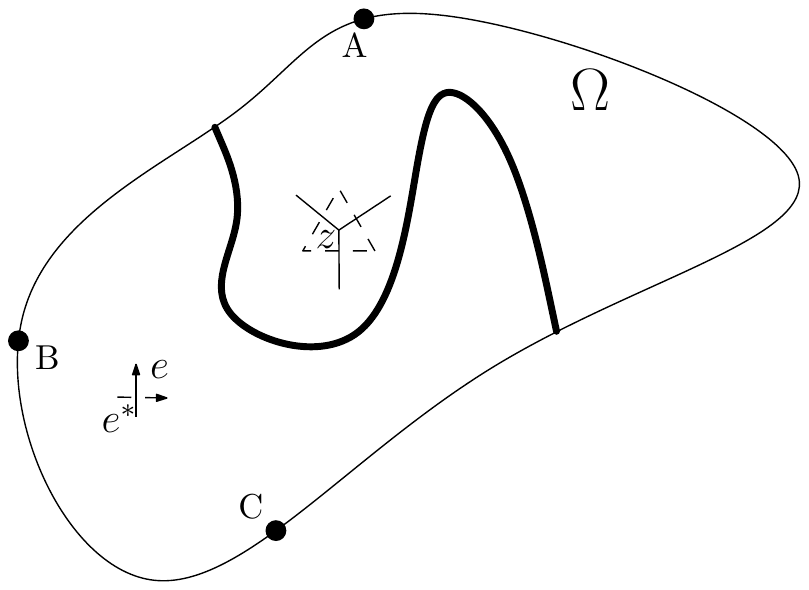}
    \end{center}
    \caption{\label{fig:smirnov1} Picture of the event
      $E_{A,\delta}(z)$. Also depicted is one oriented edge $e$ with
      its associated dual edge $e^*$. The graph $\mathbb T_\delta$ is
      drawn with dotted lines while its dual $\mathbb H_\delta$ is
      drawn with solid lines.}
  \end{figure}

  The proof consists of three steps, the second one being the most
  important:
  \begin{itemize}
  \item[(1)] Prove that $(H_{A,\delta}, H_{B,\delta},
    H_{C,\delta})_{\delta>0}$ is a precompact family of functions
    (with variable $z$).
  \item[(2)] Let $\tau=e^{2i\pi/3}$ and introduce the two
    sequences of functions defined by \begin{align*}H_\delta(z) &:=
      H_{A,\delta}(z)
      + \tau H_{B,\delta}(z) + \tau^2 H_{C,\delta}(z), \\
      S_\delta(z) &:= H_{A,\delta}(z) + H_{B,\delta}(z) +
      H_{C,\delta}(z).\end{align*} Show that any sub-sequential limits
    $h$ and $s$ of $(H_\delta)_{\delta>0}$ and $(S_\delta)_{\delta>0}$
    are holomorphic. This statement is proved using Morera's theorem,
    based on the study of discrete integrals.
  \item[(3)] Use boundary conditions to identify the possible
    sub-sequential limits $h$ and $s$. This guarantees the existence
    of limits for $(H_{A,\delta}, H_{B,\delta},
    H_{C,\delta})_{\delta>0}$. A byproduct of the proof is the exact
    computation of these limits.  \end{itemize} Then, since
  $E_{C,\delta}(D_\delta)$ is exactly the event $\mathcal C_\delta
  (\Omega,A,B,C,D)$, the limit of $H_{C,\delta}(D_\delta)$ as $\delta$
  goes to $0$ is also the limit of crossing probabilities. %This concludes the proof of Theorem~\ref{thm:cardy}.

  \paragraph{Precompactness.} We only sketch this part of the proof.
  Let $K$ be a compact subset of $\Omega$. If two points $z,z'\in K$
  are surrounded by a common open (or closed) circuit, then the events
  $E_{A,\delta}(z')$ and $E_{A,\delta}(z)$ are realized
  simultaneously. Hence, the difference
  $|H_{A,\delta}(z')-H_{A,\delta}(z)|$ is bounded above by
  \begin{align*}
    \P[z_\delta\text{~and~} z'_\delta\text{~are not surrounded by a
      common open or a closed circuit}].
  \end{align*} Let $\eta>0$ be the distance between $K$ and
  $\Omega^c$. For $z$ and $z'$, Theorem~\ref{RSW} can be applied in
  roughly $\log (|z-z'|/\eta)/\log 3$ concentric annuli, hence there
  exist two positive constants $C_K$ and $\varepsilon_K$ depending
  only on $K$ such that, for every $\delta>0$,
  \begin{equation}
    \label{eq:mainholder}
    \left| H_{A,\delta}(z') - H_{A,\delta}(z)  \right| \leqslant C_K
    \left| z'-z \right| ^{\varepsilon_K}
  \end{equation}
  and a similar bound for $H_{B,\delta}$ and $H_{C,\delta}$.
  Furthermore, similar estimates can be obtained along the boundary of
  $\Omega$ as long as we are away from $A$, $B$ and $C$.

  Since the functions are extended on the whole domain (see definition
  above), we obtain a family of uniformly Hölder maps from any compact
  subset of $\overline\Omega\setminus\{A,B,C\}$ to $[0,1]$.  By the
  Arzelà-Ascoli theorem, the family is relatively compact with respect
  to uniform convergence.  It is hence possible to extract a
  subsequence $(H_{A,\delta_n}, H_{B,\delta_n},
  H_{C,\delta_n})_{n>0}$, with $\delta_n\to0$, which converges
  uniformly on every compact to a triple of Hölder maps
  $(h_A,h_B,h_C)$ from $\overline\Omega\setminus\{A,B,C\}$ to
  $[0,1]$. From now on, we set $h=h_A+\tau h_B+\tau^2 h_C$ and
  $s=h_A+h_B+h_C$ (they are the limits of $(H_{\delta_n})_{n>0}$ and
  $(S_{\delta_n})_{n>0}$ respectively).

  \paragraph{Holomorphicity of $h$ and $s$.}

  We treat the case of $h$; the case of $s$ follows the same lines. To
  prove that $h$ is holomorphic, one can apply Morera's theorem (see
  \emph{e.g.}~\cite{Lan99}). Formally, one needs to prove that the
  integral of $h$ along $\gamma$ is zero for any simple, closed,
  smooth curve $\gamma$ contained $\Omega$. In order to prove this
  statement, we show that $(H_{\delta_n})_n$ is a sequence of (almost)
  \emph{discrete holomorphic} functions, where one needs to specify
  what is meant by discrete holomorphic. In our case, we take it to
  mean that discrete contour integrals vanish. We refer
  to~\cite{Smi10} for more details on discrete holomorphicity,
  including other definitions of it and its connections to
  statistical physics.

  Consider a simple, closed, smooth curve $\gamma$ contained in
  $\Omega$. For every $\delta>0$, let $\gamma_\delta$ be a
  discretization of $\gamma$ contained in $\Omega_\delta$, \ie\ a
  finite chain $(\gamma_\delta(k))_{0\leqslant k\leqslant N_\delta}$
  of pairwise distinct sites of $\Omega_\delta$, ordered in the
  counter-clockwise direction, such that for every index $k$,
  $\gamma_\delta(k)$ and $\gamma_\delta(k+1)$ are nearest neighbors,
  and chosen in such a way that the Hausdorff distance between
  $\gamma_\delta$ and $\gamma$ goes to $0$ with $\delta$. Notice that
  $N_\delta$ can be taken of order $\delta^{-1}$, which we shall
  assume from now on.

  For an edge $e\in \mathbb H_\delta$, define $e^*$ to be the rotation
  by $\pi/2$ of $e$ around its center (it is an edge of the triangular
  lattice). For an edge $e$ of the hexagonal lattice, let
  $$H_\delta(e)\eqd  \frac  {H_\delta(x)  + H_\delta(y)}  {2},$$
  where $e=xy$ ($x$ and $y$ are the endpoints of the edge $e$).

  An oriented edge $e^*$ of $\mathbb T_\delta$ \emph{belongs} to
  $\gamma_\delta$ if it is of the form $\gamma_\delta(k)
  \gamma_\delta(k+1)$. In such a case, we set $e^*\in\gamma_\delta$.
  Define the discrete integral $I_\gamma^\delta(H)$ of $H_\delta$ (and
  similarly $I_\gamma^\delta(S)$ for $S_\delta$) along $\gamma_\delta$
  by $$I_{\gamma}^{\delta}(H) \eqd \sum _{e^*\in \gamma} e^*
  H_{\delta} (e).$$
In the formula above, $e^*$ is considered as a vector in $\mathbb C$ of length $\delta$.

  Our goal is now to prove that $I_\gamma^\delta(H)$ and
  $I_\gamma^\delta(S)$ converge to $0$ as $\delta$ goes to $0$.
  % Since along the sequence $(\delta_n)$, they also converge to
  % $\oint_\gamma h(z) \dd z$ and $\oint_\gamma s(z) \dd z$, it will
  % imply that $h$ and $s$ are holomorphic via Morera's Theorem
  % (notice that $h$ and $s$ are continuous as uniform limits of
  % continuous functions).
  For every oriented edge $e=xy\in\mathbb H_\delta$,
  set $$P_{A,\delta}(e)~=~\P\big(E_{A,\delta}(y)\setminus
  E_{A,\delta}(x)\big),$$ and similarly $P_{B,\delta}$ and
  $P_{C,\delta}$.

  \begin{lemma}\label{representation_formula}
    For any smooth $\gamma$, as $\delta$ goes to $0$,
    \begin{align}
      \label{eq:discrint4}
      I_{\gamma}^{\delta}(H) &= \sum_{e^*\text{~surrounded
          by~}\gamma_\delta} e^*\left[ P_{A,\delta}(e) + \tau
        P_{B,\delta}(e) +
        \tau^2 P_{C,\delta}(e) \right] + o(1),\\
      I_{\gamma}^{\delta}(S) &= \sum_{e^*\text{~surrounded
          by~}\gamma_\delta} e^* \left[ P_{A,\delta}(e) +
        P_{B,\delta}(e) + P_{C,\delta}(e) \right] + o(1),
    \end{align}
    where the sum runs over oriented edges of $\mathbb T_\delta$
    surrounded by the closed curve $\gamma_\delta$.
  \end{lemma}

  \begin{proof}
    We treat the case of $H_\delta$; that of $S_\delta$ is
    similar. For every oriented edge $e=xy$ in $\mathbb H_\delta$,
    define     $$\partial_e H_\delta \eqd H_\delta(y)-H_\delta(x).$$
    If $f$ is a face of $\mathbb T_\delta$, let $\partial f$ be its
    boundary oriented in counter-clockwise order, seen as a set of
    oriented edges.  With these notations, we get the following
    identity:
    \begin{equation}
      \label{eq:discrint2}
      I_{\gamma}^{\delta}(H)~=~\sum_{e^*\in \gamma_\delta} e^*
      H_\delta(e)~=~\sum_{f\text{~surrounded by~}\gamma_\delta} \sum_{e^*\in \partial
        f} e^* H_\delta(e),
    \end{equation}
    where the first sum on the right is over all faces of $\mathbb
    T_\delta$ surrounded by the closed curve $\gamma_\delta$.  Indeed,
    in the last equality, each boundary term is obtained exactly once
    with the correct sign, and each interior term appears twice with
    opposite signs. The sum of $e^*H_\delta(e)$ around $f$ can be
    rewritten in the following fashion: \[ \sum_{e^*\in \partial f}
    e^* H_\delta(e) = \sum_{e^*=uv\in\partial f} i\left( \frac {u +
        v}{2} - f \right) \partial_e H_\delta,\] where $f$ denotes the
    complex coordinate of the center of the face $f$. Putting this
    quantity in the sum~\eqref{eq:discrint2}, the term $\partial_e
    H_\delta=H_\delta(y)-H_\delta(x)$ appears twice for $x,y\in\mathbb H_\delta$
    nearest neighbors bordered by two triangles in $\gamma_\delta$,
    and the factors $i(u+v)/2=i(x+y)/2$ cancel between the two
    occurrences (here $e^*=uv$), leaving only $i$ times the difference
    between the centers of the faces, \ie\ the complex coordinate of
    the edge $e^*$. Therefore,
    \begin{equation}
      \label{eq:discrint3}
      I_{\gamma}^{\delta}(H) = \frac12 \sum_{e^*\subset
        Int(\gamma_\delta)} e^* \partial_e H_\delta + o(1).
    \end{equation}
    In the previous equality, we used the fact that the total
    contribution of the boundary goes to $0$ with $\delta$. Indeed,
    $e^*$ is of order $\delta$, and
    \begin{equation}\label{expression}
      \partial_e H_{\delta}~=~P_{A,\delta} (e) - P_{A,\delta}
      (-e)~+~\tau(P_{B,\delta} (e) - P_{B,\delta}
      (-e))~+~\tau^2(P_{C,\delta} (e) - P_{C,\delta} (-e)),
    \end{equation}
    so that Theorem~\ref{RSW} gives a bound of $\delta^{1+\ep}$ for
    $e^*\partial_eH_\delta$ (one may for instance perform a computation similar to the one used for precompactness). Since there are roughly $\delta^{-1}$
    boundary terms, we obtain that the boundary accounts for at most
    $\delta^\ep$.

    Replacing $\partial H_\delta$ by \eqref{expression} in the
    equation \eqref{eq:discrint3}, and re-indexing the sum to obtain
    each oriented edge in exactly one term, we get the announced
    equality \eqref{eq:discrint4}.
  \end{proof}

  \begin{lemma}[Smirnov~\cite{Smi01}]\label{color_switching}
    For every three edges $e_1,e_2,e_3$ of $\Omega_\delta^*$ emanating
    from the same site, ordered counterclockwise, we have the
    following identities: $$P_{A,\delta}(e_1) = P_{B,\delta}(e_2) =
    P_{C,\delta}(e_3).$$
  \end{lemma}

  \begin{figure}[ht!]
    \begin{center}
      \includegraphics[width=0.50\textwidth]{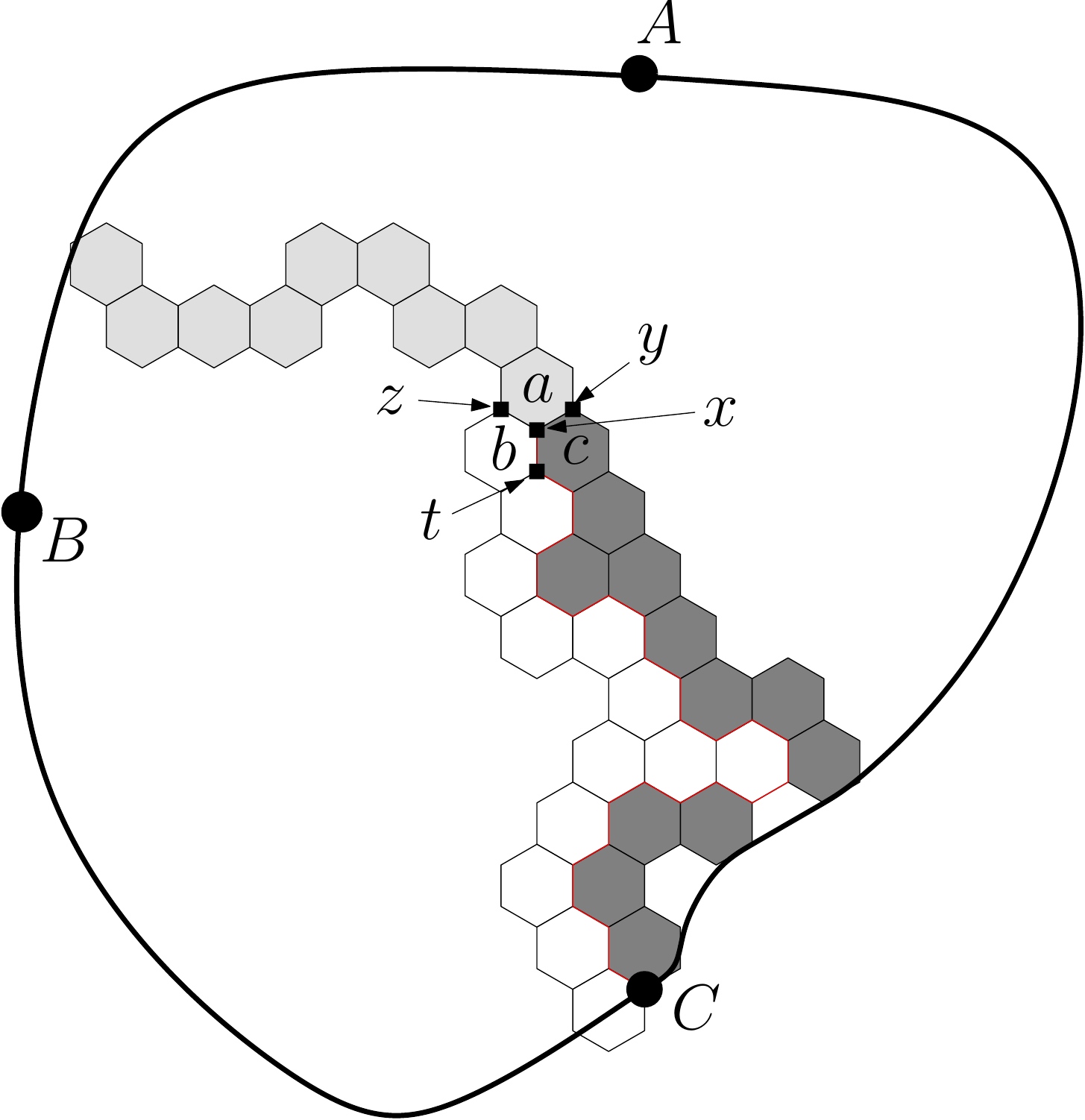}
    \end{center}
    \caption{\label{fig:smirnov2} The dark gray and the white hexagons
      are the hexagons on $\overline{\Gamma}$, $\Gamma$ being in
      black.}
  \end{figure}

  Even though we include the proof for completeness, we refer the
  reader to~\cite{Smi01} for the (elementary, but very clever) first
  derivation of this result. The lemma extends to site-percolation with
  parameter $1/2$ on any planar triangulation.

  \begin{proof}
    Index the three faces (of $\mathbb H_\delta$) around $x$ by $a$,
    $b$ and $c$, and the sites by $y$, $z$ and $t$ as depicted in
    Fig.~\ref{fig:smirnov2}.

    Let us prove that $P_{A,\delta}(e_1) = P_{B,\delta}(e_2)$. The
    event $E_{A,\delta}(y)\setminus E_{A,\delta}(x)$ occurs if and
    only if there are open paths from $AB$ to $a$ and from $AC$ to
    $c$, and a closed path from $BC$ to $b$.

    Consider the interface $\Gamma$ between the open clusters connected
    to $AC$ and the closed clusters connected to $BC$, starting at $C$,
    up to the first time it hits $x$ (it will do it if and only if
    there exist an open path from $AC$ to $c$ and a closed path from
    $BC$ to $b$). Fix a deterministic self-avoiding path of $\Omega^*_\delta$, denoted $\gamma$, from $C$ to
    $x$. The event $\{\Gamma=\gamma\}$ depends only on sites adjacent
    to $\gamma$ (we denote the set of such sites
    $\overline{\gamma}$). Now, on $\{\Gamma=\gamma\}$, there exists a
    bijection between configurations with an open path from $a$ to
    $AB$ and configurations with a closed path from $a$ to $AB$ (by
    symmetry between open and closed sites in the domain
    $\Omega_\delta\setminus \overline{\gamma}$). This is true for any
    $\gamma$ (the fact that the path is required to be self-avoiding is crucial here), hence there is a bijection between the event
    $$E_{A,\delta}(y)\setminus
    E_{A,\delta}(x)=\bigcup_{\gamma}~\{\Gamma=\gamma\} \cap
    \{a\leftrightarrow AB\text{ in }\Omega_\delta\setminus
    \overline{\gamma}\}$$ and
    $$E:=\bigcup_{\gamma}~\{\Gamma=\gamma\} \cap
    \{a\stackrel{*}{\leftrightarrow} AB\text{ in
    }\Omega_\delta\setminus \overline{\gamma}\}.$$ Note that
    $E_{B,\delta}(z)\setminus E_{B,\delta}(x)$ is the image of $E$
    after switching the states of all sites of $\mathbb T_\delta$ (or
    equivalently faces of $\mathbb H_\delta$). Hence, the two events
    are in one-to-one correspondence. Since $\P$ is
    uniform on the set of configurations,
    $$P_{A,\delta}(e_1) =\P(E_{B,\delta}(z)\setminus
    E_{B,\delta}(x))=\P(E)=P_{B,\delta}(e_2).$$ This argument is the
    key step of the lemma, and is sometimes called the
    \textbf{color-switching trick}.
  \end{proof}

  We are now in a position to prove that $I_\gamma^{\delta}(H)$ and
  $I_\gamma^{\delta}(S)$ converge to $0$. From
  Lemmas~\ref{representation_formula} and~\ref{color_switching}, we
  obtain by re-indexing the
  sum $$I_\gamma^{\delta}(H)~=~\sum_{e^*\subset Int(\gamma_\delta)}
  (e^* + \tau(\tau.e)^* + \tau^2(\tau^2.e)^*) P_A(e) + o(1)=o(1),$$
  since
  \begin{equation}\label{eq:phi0}
    e^* + \tau(\tau.e)^* + \tau^2(\tau^2.e)^*=0.
  \end{equation}
  Similarly, for $s$: $$I_\gamma^{\delta}(S)~=~\sum _{e^*
  \subset
    Int(\gamma_\delta)} (e^*+(\tau.e)^*+(\tau^2.e)^*) P_A(e) +
  o(1)=o(1).$$ Here, we have used
  \begin{equation}
    \label{eq:phi1}
    e^*  + (\tau.e)^* +(\tau^2.e)^*=0.
  \end{equation}
  This concludes the proof of the holomorphicity of $h$ and $s$.

  \paragraph{Identification of $s$ and $h$.}

  Let us start with $s$. Since it is holomorphic and real-valued, it
  is constant. It is easy to see from the boundary conditions (near a
  corner for instance) that it is identically equal to 1. Now consider
  $h$. Since $h$ is holomorphic, it is enough to identify boundary
  conditions to specify it uniquely.

  Let $z\in \Omega$. Since $h_A(z)+h_B(z)+h_C(z)=1$, $h(z)$ is a
  barycenter of $1$, $\tau$ and $\tau^2$ hence it is inside the
  triangle with vertices $1$, $\tau$ and $\tau^2$. Furthermore, if $z$
  is on the boundary of $\Omega_\delta^*$, lying between $B$ and $C$,
  $h_A(z)=0$ (using Theorem~\ref{RSW}), thus $h_B(z)+h_C(z)=1$ (since
  $s=1$).  Hence, $h(z)$ lies on the interval $[\tau,\tau^2]$ of the
  complex plane.  Besides, $h(B)=\tau$ and $h(C)=\tau^2$, so $h$
  induces a continuous map from the boundary interval $[BC]$ of
  $\Omega$ onto $[\tau, \tau^2]$.  By Theorem~\ref{RSW} yet again (more precisely Corollary~\ref{any_shape}), $h$
  is one-to-one on this boundary interval (we leave it as an exercise). Similarly, $h$ induces a
  bijection between the boundary interval $[AB]$ (resp.\ $[CA]$) of
  $\Omega$ and the complex interval $[1,\tau]$ (resp.\
  $[\tau^2,1]$). Putting the pieces together we see that $h$ is a
  holomorphic map from $\Omega$ to the triangle with vertices $1$,
  $\tau$ and $\tau^2$, which extends continuously to $\bar \Omega$ and
  induces a continuous bijection between $\partial\Omega$ and the
  boundary of the triangle.

  From standard results of complex analysis (``principle of
  corresponding boundaries'', cf.\ for instance~Theorem~4.3
  in~\cite{Lan99}), this implies that $h$ is actually a conformal map
  from $\Omega$ to the interior of the triangle.  But we know that $h$
  maps $A$ (resp.\ $B$, $C$) to $1$ (resp.\ $\tau$, $\tau^2$). This
  determines $h$ uniquely and concludes the proof of
  Theorem~\ref{thm:cardy}.
\end{proof}

% In other words, there is only one possible limit for the triple
% $(H_A,H_B,H_C)$ as $\delta$ goes to $0$, which gives conformal
% invariance for free

As a corollary of the proof, we get a nice expression for $h_A$: if
$\Phi_{\Omega,A,B,C}$ is the conformal map from $\Omega$ to the
triangle mapping $A$, $B$ and $C$ as previously (which means of course
that $\Phi_{\Omega,A,B,C}=h$) then $$H_{A,\delta}(z) \to \frac {2\Re
  e(\Phi_{\Omega,A,B,C}(z)) + 1}{3}.$$ If $\Omega$ is the equilateral
triangle itself, then $h$ is the identity map and we obtain Cardy's
formula in Carleson's form: if $D\in[CA]$
then $$f(\Omega,A,B,C,D)~=~\frac {|CD|}{|AB|}.$$

It is also to be noted that~\eqref{eq:phi0} actually characterizes the
triangular lattice (and therefore its dual, the hexagonal one), which explains why this proof works only for this lattice.

\subsection{Scaling limit of interfaces}
\label{sec:convergence_SLE}

We now show how Theorem~\ref{thm:cardy} can be used to show
Theorem~\ref{thm:SLE}. We start by recalling several properties of
$\SLE$ processes.

\subsubsection{A crash-course on Schramm--Loewner Evolutions}

In this paragraph, several non-trivial concepts about Loewner chains
are used and we refer to \cite{Law05} and \cite{Sun11} for
details. We briefly recall several useful facts in the next
paragraph.  We do not aim for completeness
(see~\cite{Law05,Wer04,Wer05} for details). We simply introduce
notions needed in the next sections. Recall that a \emph{domain} is a
simply connected open set not equal to $\mathbb C$. We first explain
how a curve between two points on the boundary of a domain can be
encoded via a real function, called the driving process. We then
explain how the procedure can be reversed. Finally, we describe the
Schramm-Loewner Evolution.

\paragraph{From curves in domains to the driving process}

Set $\mathbb H$ to be the upper half-plane. Fix a compact set $K\subset
\overline{\mathbb H}$ such that $H=\mathbb H\setminus K$ is simply connected.
Riemann's mapping theorem guarantees the existence of a conformal map from $H$
onto $\mathbb H$. Moreover, there are \emph{a   priori} three real degrees of
freedom in the choice of the conformal map, so that it is possible to fix its
asymptotic behavior as $z$ goes to $\infty$. Let $g_K$ be the unique conformal
map from $H$ onto $\mathbb H$ such that $$g_K(z)~:=~z + \frac{C}{z} + O\left(
\frac1{z^2} \right).$$ The proof of the existence of this map is not
completely obvious and requires Schwarz's reflection principle. The constant
$C$ is called the \emph{$h$-capacity} of $K$. It acts like a capacity: it is
increasing in $K$ and the $h$-capacity of $\lambda K$ is $\lambda^2$ times the
$h$-capacity of $K$.

There is a natural way to parametrize certain continuous non-self-crossing curves
$\Gamma:\mathbb R_+\rightarrow \overline{\mathbb H}$ with
$\Gamma(0)=0$ and with $\Gamma(s)$ going to $\infty$ when $s\rightarrow
\infty$. For every $s$, let $H_s$ be the connected component of
$\mathbb H\setminus\Gamma[0,s]$ containing $\infty$. We denote by
$K_s$ the \emph{hull created by} $\Gamma[0,s]$, \emph{i.e.} the
compact set $\overline{\mathbb H}\setminus H_s$.
By construction, $K_s$ has a certain $h$-capacity $C_s$. The
continuity of the curve guarantees that $C_s$ grows continuously, so that it
is possible to parametrize the curve via a time-change $s(t)$ in such a way
that $C_{s(t)}=2t$. This parametrization is called the \emph{$h$-capacity
parametrization}; we will assume it to be chosen, and reflect this by using
the letter $t$ for the time parameter from now on. Note that in general, the
previous operation is not a proper reparametrization, since any part of the
curve ``hidden from $\infty$'' will not make the $h$-capacity grow, and thus
will be mapped to the same point for the new curve; it might also be the case
that $t$ does not go to infinity along the curve (\eg\ if $\Gamma$ ``crawls''
along the boundary of the domain), but this is easily ruled out by crossing-type
arguments when working with curves coming from percolation configurations.

The curve can be encoded via the family of conformal maps $g_t$ from $H_t$ to
$\mathbb H$, in such a way that $$g_t(z)~:=~z +
\frac{2t}{z}+O\left(\frac{1}{z^2}\right).$$ Under mild conditions, the
infinitesimal evolution of the family $(g_t)$ implies the existence of a
continuous real valued function $W_t$ such that for every $t$ and $z\in H_t$,
\begin{equation}\label{fg}\partial_t
g_t(z)~=~\frac{2}{g_t(z)-W_t}.\end{equation} The function $W_t$ is called the
\emph{driving function} of $\Gamma$. The typical required hypothesis for $W$ to
be well-defined is the following \emph{Local Growth Condition}: \medbreak \noindent\emph{For any
$t\ge0$ and for any $\ep>0$, there exists $\delta>0$ such that for   any $0
\le s\le t$, the diameter of $g_s(K_{s+\delta}\setminus K_s)$   is smaller
than $\ep$.}
\medbreak
This condition is always satisfied in the case of curves (in general, Loewner chains can be defined for families of growing hulls, see \cite{Law05} for additional details).

\paragraph{From a driving function to curves} It is important to
notice that the procedure of obtaining $W$ form $\gamma$ is reversible under mild assumptions on the
driving function. We restrict our attention to the
upper half-plane.

If a continuous function $(W_t)_{t>0}$ is given, it is possible to
reconstruct $H_t$ as the set of points $z$ for which the differential
equation \eqref{fg} with initial condition $z$ admits a solution
defined on $[0,t]$. We then set $K_t=\overline{\mathbb H}\setminus
H_t$. The family of hulls $(K_t)_{t>0}$ is said to be the {\em Loewner
  Evolution with driving function} $(W_t)_{t>0}$.

So far, we did not refer to any curve in this construction. If there
exists a parametrized curve $(\Gamma_t)_{t>0}$ such that for any
$t>0$, $H_t$ is the connected component of $\mathbb H\setminus
\Gamma[0,t]$ containing $\infty$, the Loewner chain $(K_t)_{t>0}$ is
said to be {\em generated by a curve}. Furthermore, $(\Gamma_t)_{t>0}$
is called the \emph{trace} of $(K_t)_{t>0}$.
%Interestingly, the hull $(K_t)_{t>0}$ does not necessarily come  from a parametrized curve $(\Gamma_t)_{t>0}$ in the sense that $H_t$ is the connected component of $\mathbb H\setminus \Gamma[0,t]$ containing $\infty$.

A general necessary and sufficient condition for a parametrized
non-self-crossing curve in $(\Omega,a,b)$ to be the time-change of the
trace of a Loewner chain is the following:
\begin{enumerate}
\item[(C1)] Its $h$-capacity is continuous;
\item[(C2)] Its $h$-capacity is strictly increasing;
\item[(C3)] The hull generated by the curve satisfies the Local Growth
  Condition.
\end{enumerate}

\paragraph{The Schramm-Loewner Evolution} We are now in a position to
define Schramm--Loewner Evolutions:

\begin{definition}[$\SLE$ in the upper half-plane]
  The chordal Schramm--Loewner Evolution in $\mathbb H$ with parameter
  $\kappa>0$ is the (random) Loewner chain with driving process
  $W_t:=\sqrt \kappa B_t$, where $B_t$ is a standard Brownian motion.
\end{definition}

Loewner chains in other domains are easily defined via conformal maps:

\begin{definition}[$\SLE$ in a general domain]
  Fix a domain $\Omega$ with two points $a$ and $b$ on the boundary
  and assume it has a nice boundary (for instance a Jordan curve). The
  chordal Schramm--Loewner evolution with parameter $\kappa>0$ in
  $(\Omega,a,b)$ is the image of the Schramm--Loewner evolution in the
  upper half-plane by a conformal map from $(\mathbb H,0,\infty)$ onto
  $(\Omega,a,b)$.
\end{definition}

The scaling properties of Brownian motion ensure that the definition
does not depend on the choice of the conformal map involved;
equivalently, the definition is consistent in the case $\Omega=\mathbb
H$. Defined as such, SLE is a random family of growing hulls, but it
can be shown that the Loewner chain is generated by a curve (see
\cite{RS05} for $\kappa\ne 8$ and \cite{LSW04} for $\kappa=8$).

% \begin{theorem}[\cite{RS05}, \cite{LSW04} for $\kappa=8$]
%   For $\kappa>0$, there exists a random continuous curve
%   $(\Gamma_t)_{t>0}$, called the {\em trace of SLE($\kappa$)}, such
%   that for any $t>0$, the hull $K_t$ of SLE($\kappa$) is given by
%   the complement of $\mathbb H$ minus the connected component of
%   $\mathbb H\setminus\Gamma[0,t]$ containing $\infty$.
% \end{theorem}

\paragraph{Markov domain property and SLE} To conclude this section,
let us justify the fact that SLE traces are natural scaling limits for
interfaces of conformally invariant models. In order to explain this
fact, we need the notion of \emph{domain Markov property} for a family
of random curves. Let $(\Gamma_{(\Omega,a,b)})$ be a family of random curves from $a$ to $b$ in $\Omega$, indexed by domains $(\Omega,a,b)$.

\begin{definition}[Domain Markov property]\label{def:markov_property}
  A family of random continuous curves $\Gamma_{(\Omega,a,b)}$
  in simply connected domains is said
  to satisfy the domain Markov property if for every $(\Omega,a,b)$
  and every $t>0$, the law of the curve $\Gamma_{(\Omega,a,b)} [t,\infty)$
  conditionally on $\Gamma_{(\Omega,a,b)} [0,t]$ is the same as the law of
  $\Gamma_{(\Omega_t,\Gamma_t,b)}$, where $\Omega_t$ is the connected
  component of $\Omega\setminus\Gamma_t$ having $b$ on its boundary.
\end{definition}

Discrete interfaces in many models of statistical physics naturally
satisfy this property (which can be seen as a variant of the
Dobrushin-Lanford-Ruelle conditions for Gibbs measures, \cite{Geo88}),
and therefore their scaling limits, provided that they exist, also
should. Schramm proved the following result in~\cite{Sch00}, which in some
way justifies $\SLE$ processes as the only natural candidates for such scaling
limits:

\begin{theorem}[Schramm~\cite{Sch00}]
  Every family of random curves $\Gamma_{(\Omega,a,b)}$ which
  \begin{itemize}
  \item is conformally invariant,
  \item satisfies the domain Markov property, and
  \item satisfies that $\Gamma_{(\mathbb H,0,\infty)}$ is scale
    invariant,
  \end{itemize}
  is the trace of a chordal Schramm--Loewner evolution with parameter
  $\kappa\in[0,\infty)$.
\end{theorem}

\begin{remark}
  It is formally not necessary to assume scale invariance of the curve
  in the case of the upper-half plane, because it can be seen as a
  particular case of conformal invariance; we keep it nevertheless in
  the previous statement because it is potentially easier, while still
  informative, to prove.
\end{remark}

\subsubsection{Strategy of the proof of Theorem~\ref{thm:SLE}}

In the following paragraphs, we fix a simply-connected domain $\Omega$
with two points $a$ and $b$ on its boundary. We consider percolation
with parameter $p=1/2$ on a discretization $\Omega_\delta$ of $\Omega$
by the rescaled triangular lattice $\delta \mathbb T$. Let $a_\delta$
and $b_\delta$ be two boundary sites of $\Omega_\delta^*$ near $a$ and
$b$ respectively. As explained in the introduction, the boundary of
$\Omega_\delta$ can be divided into two arcs $a_\delta b_\delta$ and
$b_\delta a_\delta$. Assuming that the first arc is composed of open
sites, and the second of closed sites, we obtain a unique interface
defined on $\Omega_\delta^*$ between the open cluster connected to
$a_\delta b_\delta$, and the closed cluster connected to $b_\delta
a_\delta$. This path is denoted by $\gamma_\delta$ and is called the
\emph{exploration path}.

The strategy to prove that $(\gamma_\delta)$ converges to the trace of
$\SLE(6)$ follows three steps:
\begin{itemize}
\item First, prove that the family $(\gamma_\delta)$ of curves is tight.
\item Then, show that any sub-sequential limit can be reparametrized in
  such a way that it becomes the trace of a Loewner evolution with a
  continuous driving process.
\item Finally, show that the only possible driving process for the
  sub-sequential limits is $\sqrt{6}B_t$ where $B_t$ is a standard
  Brownian motion.
\end{itemize}
The main step is the third one. In order to identify Brownian motion
as the only possible driving process for the curve, we find computable
quantities expressed in terms of the limiting curve. In our case,
these quantities will be the limits of certain crossing
probabilities.  The fact that these (explicit) functions are
martingales implies martingale properties of the driving
process. L\'evy's theorem (which states that a continuous real-valued process $X$ such
that both $X_t$ and $X_t^2-6t$ are martingales is necessarily of the form
$\sqrt 6 B_t$) then gives that the driving process must be $\sqrt 6
B_t$.

\subsubsection{Tightness of interfaces}

Recall that the convergence of random parametrized curves (say with
time-parameter in $\mathbb R$) is in the sense of the \textbf{weak
  topology} inherited from the following distance on curves:
\begin{eqnarray}
  d(\Gamma,\tilde\Gamma)~=~\inf_{\phi} \sup_{u\in\mathbb R}
  |\Gamma(u)-\tilde \Gamma(\phi(u))|,
\end{eqnarray}
where the infimum is taken over all reparametrizations (\emph{i.e.}
strictly increasing continuous functions
$\phi\colon\mathbb R\rightarrow\mathbb R$ with $\phi(0)=0$ and $\phi$ tends to infinity as $t$ tends to infinity).

In this section, the following theorem is proved:

\begin{theorem}\label{compactness_interface}
  Fix a domain $(\Omega,a,b)$. The family $(\gamma_\delta)_{\delta>0}$
  of exploration paths for critical percolation in $(\Omega,a,b)$ is
  tight.
\end{theorem}

The question of tightness for curves in the plane has been studied in
the milestone paper \cite{AB99}. In this paper, it is proved that a
sufficient condition for tightness is the absence, on every scale, of
annuli crossed back and forth an arbitrary large number of times.

For $\delta>0$, let $\mu_\delta$ be the law of a random path
$\Gamma_\delta$ on $\Omega_\delta$ from $a_\delta$ to $b_\delta$. For
$x\in \Omega$ and $r<R$, let $\Lambda_r(x)=x+\Lambda_r$ and
$S_{r,R}(x)=\Lambda_R(x)\setminus \Lambda_r(x)$ and define $\mathcal
A_k(x;r,R)$ to be the event that there exist $k$ disjoint sub-paths of
the curve $\Gamma_\delta$ crossing between the outer and inner
boundaries of $S_{r,R}(x)$.

\begin{theorem}[Aizenman-Burchard \cite{AB99}]\label{Aizenman-Burchard}
  Let $\Omega$ be a simply connected domain and let $a$ and $b$ be two
  marked points on its boundary. For $\delta>0$, let $\Gamma_{\delta}$
  denote a random path on $\Omega_\delta$ from $a_\delta$ to
  $b_\delta$ with law $\mu_\delta$.

  If there exist $k\in \mathbb N$, $C_k<\infty$ and $\Delta_k>2$ such
  that for all $\delta<r<R$ and $x\in \Omega$,
  \begin{equation*}
    \mu_{\delta}(\mathcal A_k(x;r,R))\leq C_k
    \Big(\frac{r}{R}\Big)^{\Delta_k},
  \end{equation*}
  then the family of curves $(\Gamma_{\delta})$ is tight.
\end{theorem}

We now show how to exploit this theorem in order to prove
Theorem~\ref{compactness_interface}. The main tool is
Theorem~\ref{RSW}.

\begin{proof}[Proof of Theorem~\ref{compactness_interface}]
  Fix $x\in \Omega$, $\delta<r<R$ and recall that the lattice has mesh
  size $\delta$. Let $k$ be a positive integer to be fixed later. By the Reimer inequality (recall that the Reimer inequality is simply the BK inequality for non-increasing events),
  \begin{equation*}
    \Pp\big(\mathcal A_{k}(x;r,3r)\big)\leq
    \left[\Pp\big(\mathcal A_1(x;r,3r)\big)\right]^k.
  \end{equation*}
  Using Theorem~\ref{RSW}, $\Pp\big(\mathcal A_1(x;r,3r))\leq
  1-\Pp(\mathcal E_{r/\delta})<1-C$, where $\mathcal E_n$ is the event
  that there exists a closed circuit surrounding the annulus in
  $\Lambda_{3n}\setminus\Lambda_n$. Let us fix $k$ large enough so
  that $(1-C)^k<1/27$. The annulus $S_{r,R}(x)$ can be decomposed into
  roughly $\ln_3(R/r)$ annuli of the form $S_{3^\ell
    r,3^{\ell+1}r}(x)$. For this value of $k$,
  \begin{equation}\label{estimate_AB}
    \Pp(\mathcal A_{k}(x;r,R))\leq C\left(\frac rR\right)^3,
  \end{equation}
  for some constant $C>0$. Hence, Theorem~\ref{Aizenman-Burchard}
  implies that the family $(\gamma_{\delta})$ is tight.
\end{proof}

\subsubsection{Sub-sequential limits are traces of Loewner chains}

In the previous paragraph, exploration paths (and therefore their
traces, since they coincide) were shown to be tight. Let us consider a
sub-sequential limit. We would like to show that, properly
reparametrized, the limiting curve is the trace of a Loewner chain.

\begin{theorem}\label{Loewner}
  Any sub-sequential limit of the family $(\gamma_\delta)_{\delta>0}$
  of exploration paths is almost surely the time-change of the trace
  of a Loewner chain.
\end{theorem}

The discrete curves $\gamma_\delta$ are random Loewner chains, but
this does not imply that sub-sequential limits are. Indeed, not every
continuous non-self-crossing curve can be reparametrized as the trace
of a Loewner chain, especially when it is fractal-like and has many
double points. We therefore need to provide an additional ingredient.

Condition C1 of the previous section is easily seen to be automatically satisfied by
continuous curves. Similarly, Condition C3 follows from the two others
when the curve is continuous, so that the only condition to check is
Condition C2.

This condition can be understood as being the fact that the tip of the
curve is visible from $b$ at every time. In other words, the family of
hulls created by the curve is strictly increasing. This is the case if
the curve does not enter long fjords created by its past at every
scale, see Fig.~\ref{fig:six_arm_event}.

\begin{figure}[ht!]
  \begin{center}
    \includegraphics[width=0.35\textwidth]{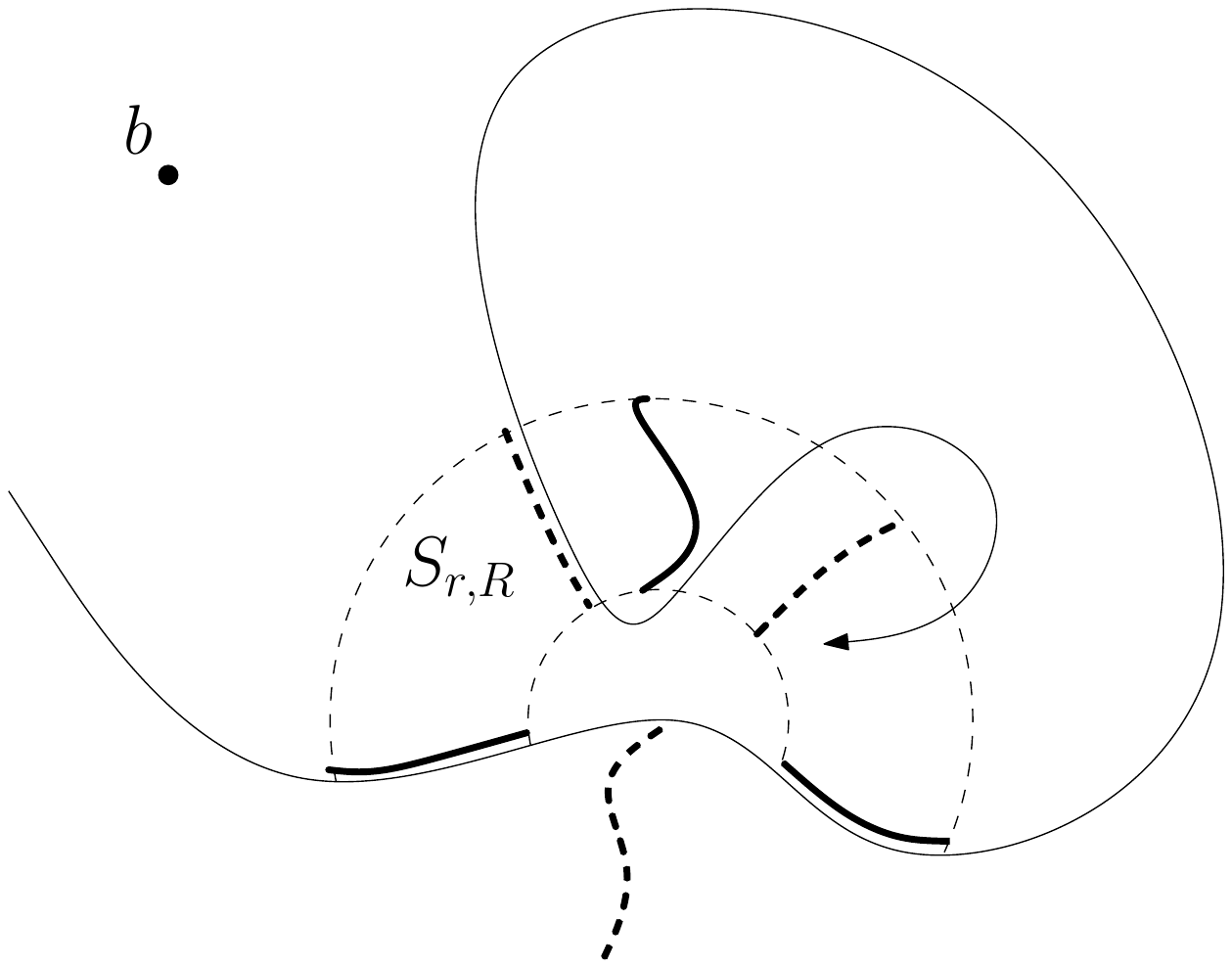}
    $\quad\quad\quad$
    \includegraphics[width=0.55\textwidth]{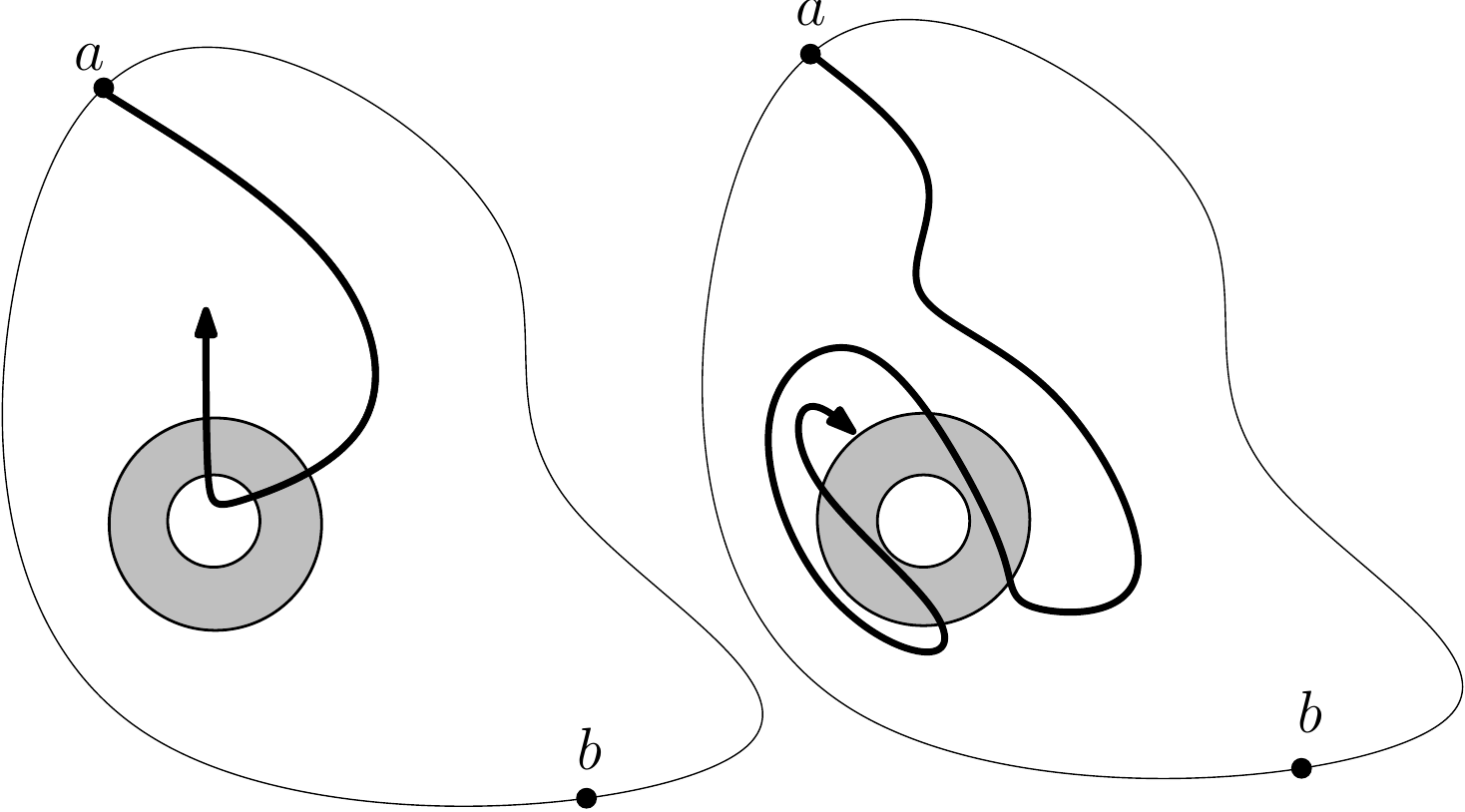}
  \end{center}
  \caption{\label{fig:six_arm_event}\textbf{Left:} An example of a
    fjord. Seen from $b$, the $h$-capacity (roughly speaking, the
    size) of the hull does not grow much while the curve is in the
    fjord. The event involves six alternating open (in plain lines) and closed (in dotted lines)
    crossings of the annulus. \textbf{Right:} Conditionally on the
    beginning of the curve, the crossing of the annulus is unforced on
    the left, while it is forced on the right.}
\end{figure}

% In the case of percolation, this corresponds to the six-arm events,
% and it boils down to proving that $\Delta_6>2$. We will prove this
% result in Proposition~\ref{6-4_arm_exponent}, and we show at the end
% of this subsection how it indeed implies that scaling limits are
% Loewner chains. Before that, we present a more general criterion
% characterizing Loewner chains.

Recently, Kemppainen and Smirnov proved a ``structural theorem''
characterizing sequences of random discrete curves
whose limit satisfies Condition C2 almost surely. This theorem generalizes
Theorem~\ref{Aizenman-Burchard}, in the sense that the condition is weaker and the conclusion stronger.  Before stating the theorem, we need
a definition. Fix $\Omega$ and two boundary points $a$ and $b$ and
consider a curve $\Gamma$. A sub-path $\Gamma[t_0,t_1]$ of a
continuous curve $\Gamma$ is called a \emph{crossing} of the annulus
$S_{r,R}(x)$ if $\Gamma_{t_0}\in\partial\Lambda_r(x)$ and
$\Gamma_{t_1}\in\partial\Lambda_R(x)$, where $t_0<t_1$ or $t_1<t_0$
and $\partial \Lambda_r$ is the boundary of $\Lambda_r$. A crossing is
called \emph{unforced} if there exists a path $\tilde \Gamma$ from $a$
to $b$ not intersecting $\Lambda_R(x)$.

\begin{theorem}[Kemppainen-Smirnov, \cite{KS10}]\label{thm:KS}Let
  $(\Omega,a,b)$ be a domain with two points on the boundary. For
  $\delta>0$, $\Gamma_\delta$ is a random continuous curve on
  $(\Omega_\delta,a_\delta,b_\delta)$ with law $\mu_\delta$.

  If there exist $C>1$ and $\Delta>0$ such that for any
  $0<\delta<r<R/C$ and for any stopping time $\tau$,
  $$\mu_\delta(\gamma_\delta[\tau,\infty]\text{ contains an unforced
    crossing in $\Omega\setminus\gamma_\delta[0,\tau]$ of $S_{r,R}(x)$})\le
  C\left(\frac{r}{R}\right)^\Delta$$ for any annulus $S_{r,R}(x)$,
   %not containing $\gamma_\delta(\tau)$
   then the family
  $(\Gamma_\delta)_{\delta>0}$ is tight and any sub-sequential limit
  can almost surely be reparametrized as the trace of a Loewner
  chain.
\end{theorem}

% We are now in a position to prove Theorem~\ref{Loewner}.
% \begin{proof}[Proof of Theorem~\ref{Loewner}]
%   Following a proof very similar to the argument leading to
%   Theorem~\ref{Aizenman-Burchard}, Theorem~\ref{RSW} implies the
%   condition of Theorem~\ref{thm:KS}.

%   % Roughly speaking, the previous condition is a uniform bound on
%   % unforced crossings. Note that it is necessary to precise the fact
%   % that the crossing is unforced.
%   %
%   % \begin{theorem}\label{Stas-Anti}
%   %   If a family of curves $(\gamma_\delta)$ satisfies Condition
%   %   $(*)$,
%   %   then it is tight. Moreover, any sub-sequential limit is a
%   %   time-changed Loewner chain.
%   % \end{theorem}
%   %%
%   %% Tightness is almost obvious, since Condition $(*)$ implies the
%   %% hypothesis in Aizenman-Burchard's theorem. The hard part is the
%   %% proof that Condition $(*)$ guarantees that the $h$-capacity of
%   %% sub-sequential limits is strictly increasing and that they create
%   %% Loewner chains generated by a curve. The reader is referred
%   %% to~\cite{KS10} for a proof of this statement.
%   %
We do not prove this theorem and refer instead to the original article
for a complete account. Theorem~\ref{RSW} implies the hypothesis of
the previous theorem following the same lines as in the proof of
Theorem~\ref{compactness_interface}. As a consequence,
Theorem~\ref{Loewner} follows readily.
% \end{proof}

In order to show Theorem~\ref{Loewner} in the case of percolation, one can run an alternative argument based
on Corollary~\ref{any_shape} and the so-called 6-arm event. This argument has already been
described precisely in \cite{Wer09}. For this reason, we do not repeat
it here and refer to these lecture notes for details.

\subsubsection{Convergence of exploration paths to $\SLE(6)$}

Fix a topological triangle $(\Omega,A,B,C)$, \emph{i.e.} a domain
$\Omega\ne \mathbb C$ delimited by a non-intersecting continuous curve
and three distinct points $A$, $B$ and $C$ on its boundary, indexed in
counter-clockwise order. Let
$(\Omega_\delta,A_\delta,B_\delta,C_\delta)$ be a discrete
approximation of $(\Omega,A,B,C)$ and $z_\delta\in
\Omega_\delta^*$. Recall the definition of $E_{A,\delta}(z_\delta)$
used in the proof of Theorem~\ref{thm:cardy}: it is the event that
there exists a non-self-intersecting path of open sites in
$\Omega_\delta$, separating $A_\delta$ and $z_\delta$ from $B_\delta$
and $C_\delta$. For technical reasons, we keep the dependency on the
domain in the notation for the duration of this section, and we set
$E_{\Omega_\delta,A_\delta,B_\delta,C_\delta}(z_\delta):=E_{A,\delta}(z_\delta)$. Also
define
$$H_n(\Omega_\delta,A_\delta,B_\delta,C_\delta,z_\delta) := \P(E_{\Omega_\delta\setminus\gamma[0,n],\gamma_n,B_\delta,C_\delta}
(z_\delta)).$$

\begin{lemma}
  For any $(\Omega,A,B,C)$, and for any $z\in \Omega$ and $\delta>0$,
  the function
  $(H_n(\Omega_\delta,A_\delta,B_\delta,C_\delta,z_\delta))_{n\ge 0}$
  is a martingale with respect to $(\mathcal F_n)_{n\ge 0}$, where
  $\mathcal F_n$ is the $\sigma$-algebra generated by the the first
  $n$ steps of $\gamma_\delta$.
\end{lemma}

\begin{proof}
  The slit domain created by ``removing'' the first $n$ steps of the
  exploration path is again a topological triangle. Conditionally on
  the $n$ first steps of $\gamma_\delta$, the law of the configuration
  in the new domain is exactly percolation in
  $\Omega\setminus\gamma_\delta[0,n]$. This observation implies that
  $H_n(\Omega_\delta,A_\delta,B_\delta,C_\delta,z_\delta)$ is the
  random variable
  $1_{E_{\Omega_\delta,A_\delta,B_\delta,C_\delta}(z_\delta)}$
  conditionally on $\mathcal F_n$, therefore it is automatically a
  martingale.
\end{proof}

\begin{proposition}\label{identification}
  Any sub-sequential limit of $(\gamma_\delta)_{\delta>0}$ which is the
  trace of a Loewner chain is the trace of $\SLE(6)$.
\end{proposition}

\begin{proof}
  Once again, we only sketch the proof in order to highlight the important steps. Consider a sub-sequential limit $\gamma$ in the domain $(\Omega,A,B)$
  which is a Loewner chain. Let $\phi$ be a map from $(\Omega,A,B)$ to
  $(\mathbb H,0,\infty)$. Our goal is to prove that $\tilde{\gamma} :=
  \phi(\gamma)$ is a chordal $\SLE(6)$ in the upper half-plane.

  Since $\gamma$ is assumed to be a Loewner chain, $\tilde \gamma$ is
  a growing hull from $0$ to $\infty$; we can assume that it is
  parametrized by its $h$-capacity. Let $W_t$ be its continuous
  driving process. Also define $g_t$ to be the conformal map from
  $\mathbb H\setminus \tilde{\gamma}[0,t]$ to $\mathbb H$ such that
  $g_t(z) = z + 2t/z + O(1/z^2)$ when $z$ goes to infinity.

  Fix $C \in \partial \Omega$ and $Z\in \Omega$. For $\delta>0$,
  recall that $H_n(\Omega_\delta, A_\delta, B_\delta, C_\delta,
  Z_\delta)$ is a martingale for $\gamma_\delta$. Since the martingale
  is bounded, $H_{\tau_t}(\Omega_\delta, A_\delta, B_\delta, C_\delta,
  Z_\delta)$ is a martingale with respect to $\mathcal F_{\tau_t}$,
  where $\tau_t$ is the first time at which $\phi(\gamma_\delta)$ has
  a $h$-capacity larger than $t$. Since the convergence of
  $\gamma_\delta$ to $\gamma$ is uniform on every compact subset of
  $(\Omega,A,B)$, one can see (with a little bit of work) that
  $$H_t(Z):=\lim_{\delta\rightarrow
    0}H_{\tau_t}(\Omega_\delta, A_\delta, B_\delta, C_\delta,
  Z_\delta)$$ is a martingale with respect to $\mathcal G_t$, where
  $\mathcal G_t$ is the $\sigma$-algebra generated by the curve
  $\tilde \gamma$ up to the first time its $h$-capacity exceeds
  $t$. By definition, this time is $t$, and $\mathcal G_t$ is the
  $\sigma$-algebra generated by $\tilde \gamma[0,t]$. In other words,
  it is the natural filtration associated with the driving process
  $(W_t)$.

  We borrow the definitions of $h_A$ and $h$ from the proof of the
  Cardy--Smirnov formula. By first mapping $\Omega$ to $\mathbb H$ and
  then applying the Cardy-Smirnov formula, we find \[ H_t(Z) = h_A
  \left( \frac {g_t(z) - W_t} {g_t(c) - W_t} \right), \] where we
  define $z := \phi(Z)$ and $c := \phi(C)$. This is a martingale for
  every choice of $z$ and $c$, so we get the family of identities \[
  \mathbb E \left[ h_A \left( \frac {g_t(z) - W_t} {g_t(c) - W_t}
    \right) \middle| \mathcal G_s \right] = h_A \left( \frac {g_s(z) -
      W_s} {g_s(c) - W_s} \right) \] for all $z \in \mathbb H$, $c \in
  \mathbb R$ and $0<s<t$ such that $z$ and $c$ are both within the
  domain of definition of $g_t$. Now, we would like to express the
  previous equality in terms of $h$ instead of $h_A$ (recall that
  $h_A=\frac13(2\Re e(h)+1)$). Noting that the two functions below, as
  functions of $z$, are holomorphic and equal at $c$, we obtain
  \[ \mathbb E \left[ h \left( \frac
      {g_t(z) - W_t} {g_t(c) - W_t} \right) \middle| \mathcal G_s
  \right] = h\left( \frac {g_s(z) - W_s} {g_s(c) - W_s} \right). \]

  We know the asymptotic expansion of $g_s$ and $g_t$ around infinity,
  so the above becomes
  \begin{equation}\label{eq:martingale}
    \mathbb E \left[ h
      \left( \frac {z - W_t + 2t/z + O(1/c^2)} {c - W_t + 2t/c + O(1/z^2)}
      \right) \middle| \mathcal G_s \right] = h \left( \frac {z - W_s +
        2s/z + O(1/z^2)} {c - W_s + 2s/c + O(1/c^2)} \right).
  \end{equation}
  Letting $z$ and $c$ go to infinity with fixed ratio $z/c = \lambda
  \in \mathbb H$, we have
  \begin{align*}
    h \left( \frac {z - W_s + 2s/z + O(1/z^2)} {c - W_s + 2s/c +
        O(1/c^2)} \right) &= h \left( \frac {\lambda - W_s/c +
        2s/\lambda c^2 + O(1/c^2)} {1 - W_s/c + 2s/c^2 + O(1/c^3)}
    \right) \\
    & \hspace{-5cm} = h \left( \lambda + \frac {(\lambda-1) W_s} c +
      \frac {(\lambda-1) W_s^2 + 2(1-\lambda^2)s/\lambda} {c^2} +
      O(c^{-3}) \right). \\
    & \hspace{-5cm} = h(\lambda) + \frac {(\lambda-1) h'(\lambda) W_s}
    {c} \\
    & \hspace{-3.6cm} + \frac {(\lambda-1)W_s^2 [h'(\lambda) +
      (\lambda-1) h''(\lambda)/2]+ 2 (1-\lambda^2) s
      h'(\lambda)/\lambda } {c^2} + O(c^{-3}).
  \end{align*}
  Using this expansion on both sides of~\eqref{eq:martingale} and
  matching the terms, we obtain two identities for $(W_t)$: \[ \mathbb
  E [W_t | \mathcal G_s] = W_s, \quad E[ W_t^2 | \mathcal G_s ] =
  W_s^2 + \frac {4 (1+\lambda) h'(\lambda) / \lambda} {2 h'(\lambda) +
    (\lambda-1) h''(\lambda)} (t-s). \]

  The function $h$ is a conformal map from the upper-half plane to the
  equilateral triangle, sending $0$, $1$ and $\infty$ to the vertices
  of the triangle; up to (explicit) additive and multiplicative
  constants $A$ and $B$, it can be written using the
  Schwarz-Christoffel formula as $$h(\lambda) =A \int^\lambda
  [z(1-z)]^{-2/3} \; \dd z+B.$$ From this, one obtains $h'(\lambda)=A
  [\lambda (1-\lambda)]^{-2/3}$ and $$\frac {h''(\lambda)}
  {h'(\lambda)} = - \frac23 \left( \frac 1 \lambda - \frac 1
    {1-\lambda}\right) = \frac {2(2\lambda-1)} {3\lambda
    (1-\lambda)}.$$ Plugging this into the previous expression shows
  that the coefficient of $(t-s)$ is identically equal to 6, and since
  we know that $(W_t)$ is a continuous process, L\'evy's theorem
  implies that it is of the form $(\sqrt6 B_t)$ where $(B_t)$ is a
  standard real-valued Brownian motion. This implies that $\gamma$ is
  the trace of the $\SLE(6)$ process in $(\Omega,A,B)$.
\end{proof}

\begin{proof}[Proof of Theorem~\ref{thm:SLE}]
  By Theorem~\ref{compactness_interface}, the family of exploration
  processes is tight. Using Theorem~\ref{Loewner}, any sub-sequential
  limit is the time-change of the trace of a Loewner chain. Consider
  such a sub-sequential limit and parametrize it by its $h$-capacity.
  Proposition~\ref{identification} then implies that it is the trace
  of SLE(6). The possible limit being unique, we are done.
\end{proof}

\section{Critical exponents}

To quantify connectivity properties at $p=1/2$, we introduce the
notion of \emph{arm-event}. Fix a sequence $\sigma\in\{0,1\}^j$ of $j$
colors (open $1$ or closed $0$). For $1\le n<N$, define $A_\sigma(n,N)$ to
be the event that there are $j$ \emph{disjoint} paths from $\partial
\Lambda_n$ to $\partial \Lambda_N$ with colors $\sigma_1$, \ldots,
$\sigma_j$ where the paths are indexed in counter-clockwise order. We
set $A_\sigma(N)$ to be $A_\sigma(k,N)$ where $k$ is the smallest
integer such that the event is non-empty. For instance, $A_1(n,N)$ is
the one-arm event corresponding to the existence of an open crossing from
the inner to the outer boundary of $\Lambda_N\setminus \Lambda_n$.

An adaptation of Corollary~\ref{critical_exponent} implies that there
exist $\alpha'_\sigma$ and $\beta'_\sigma$ such that
$$(n/N)^{\alpha'_\sigma}~\le~\P[A_\sigma(n,N)]~\le~(n/N)^{\beta'_\sigma}.$$
It is therefore natural to predict that there exists a \emph{critical
  exponent} $\alpha_\sigma\in(0,\infty)$ such that
$$\P[A_\sigma(n,N)]~=~(n/N)^{\alpha_\sigma+o(1)},$$
where $o(1)$ is a quantity converging to 0 as $n/N$ goes to 0. The
quantity $\alpha_\sigma$ is called an \emph{arm-exponent}.
We now explain how these exponents can be computed.

% Write $u_{p,n}\asymp v_{p,n}$ if there exist two constants
% $0<A,B<\infty$ not depending on $p$ nor $n$ such that $Au_{p,n}\le
% v_{p,n}\le Bu_{p,n}$ for every $p$ and $n$ in question. We also
% write $u_{p,n}\lesssim v_{p,n}$ if there exists $0<B<\infty$ such
% that $u_{p,n}\le Bv_{p,n}$ for every $p$ and $n$ in question.

\subsection{Quasi-multiplicativity of the probabilities of arm-events}

Let us start by a few technical yet crucial statements on probabilities of arm-events. These statements will be instrumental in all the following proofs.
\begin{theorem}[Quasi-multiplicativity]\label{quasi-multiplicativity}
  Fix a color sequence $\sigma$. There exists $c\in(0,\infty)$
  such that \small $$c \P  \big[ A_{\sigma} (n_1,n_2)
  \big]\P \big[A_{\sigma}(n_2,n_3)\big]\le
  \P\big[A_{\sigma}(n_1,n_3)\big] \le \P
  \big[ A_{\sigma} (n_1,n_2) \big]
  \P\big[A_{\sigma}(n_2,n_3)\big]$$\normalsize for every
  $n_1<n_2<n_3$.
\end{theorem}

The inequality $$\P \big[A_{\sigma}(n_1,n_3)\big] \le \P
\big[A_{\sigma}(n_1,n_2)\big] \P\big[A_{\sigma}(n_2,n_3)\big].$$ is
straightforward using independence. The other one is slightly more
technical. Let us mention that in the case of one arm ($\sigma=1$), or
more generally if all the arms are to be of the same color, the proof
is fairly easy (we recommend it as an exercise; see
Fig.~\ref{fig:gluing} for a hint). For general $\sigma$, the proof
requires the notion of well-separated arms. We do not discuss this
matter here and refer to the well-documented literature
\cite{Kes87,Nol08}.

\begin{figure}[ht!]
  \begin{center}
    \includegraphics[width=0.40\textwidth]{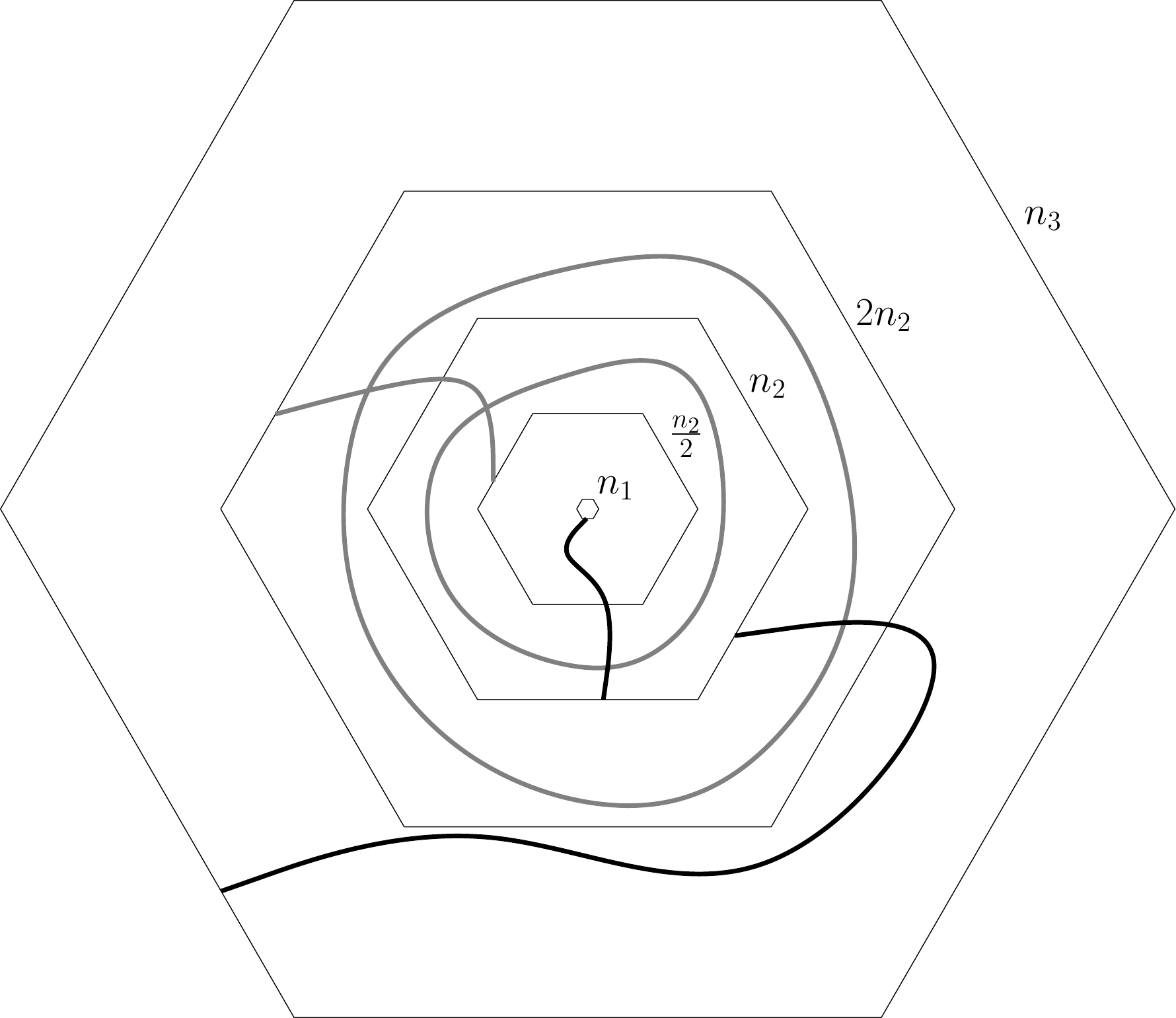}
  \end{center}
  \caption{\label{fig:gluing}The paths in the annuli
    $\Lambda_{n_3}\setminus \Lambda_{n_2}$ and $\Lambda_{n_2}\setminus
    \Lambda_{n_1}$ are in black. A combination of two circuits
    connected by a path (in gray) connects the paths together. The combination of gray paths in the middle occurs with probability bounded away from 0 thanks to
    crossing estimates (Theorem~\ref{RSW} and Corollary~\ref{any_shape}).}
\end{figure}

\medbreak

Another important tool, which is also a consequence of the
well-separation of arms, is the following localization of arms. Let
$\delta>0$; for a sequence $\sigma$ of length $j$, consider $2j+1$
points $x_1,x_2,\dots,x_{2j},x_{2j+1}=x_1$ found in clockwise order on
the boundary of $\Lambda_n$, with the additional condition that
$|x_{k+1}-x_k|\ge \delta n$ for any $k\le 2j$. Similarly, consider
$2j+1$ points $y_1,\dots,y_{2j},y_{2j+1}=y_1$ found in clockwise order
on the boundary of $\Lambda_N$, with the additional condition that
$|y_{k+1}-y_k|\ge \delta N$ for any $k\le 2j$. The sequence of
intervals $(I_k=[x_{2k-1},x_{2k}])_{k\le j}$ and
$(J_k=[y_{2k-1},y_{2k}])_{k\le j}$ are called {\em $\delta$-well
  separated landing sequences}. Let $A^{I,J}_{\sigma}(n,N)$ be the
event that for each $k$ there exists an arm of color $\sigma_k$ from
$I_k$ to $J_k$ in $\Lambda_n\setminus\Lambda_N$, these arms being
pairwise disjoint. This event corresponds to the event $A_\sigma(n,N)$
where arms are forced to start and finish in some prescribed areas of
the boundary.

\begin{proposition}\label{landing}
  Let $\sigma$ be a sequence of colors; for any $\delta>0$ there exists
  $C_\sigma<\infty$ such that, for any $2n\le N$ and any choice of
  $\delta$-well separated landing sequences $I,J$ at radii $n$ and $N$,
  $$\P \big[A^{I,J}_\sigma(n,N)\big]\le
  \P\big[A_\sigma(n,N)\big] \le C_\sigma
  \P\big[A^{I,J}_\sigma(n,N)\big].$$
\end{proposition}

Once again, only the second inequality is non trivial. We refer to
\cite{Nol08} for a comprehensive study.

\subsection{Universal arm exponents}

Before dealing with the computation of arm-exponents using SLE
techniques, let us mention that several exponents can be computed
without this elaborated machinery. These exponents, called
\emph{universal exponents}, are expected to be the same for a large
class of models, including the so-called random-cluster models with
cluster weights $q\le 4$ (see \cite{Gri06} for a review on the
random-cluster model).  In order to state the result, we need to
define arm events in the half-plane. Let $\mathbb H^+$ be the set of
vertices in $\mathbb H$ with positive second coordinate. For a color
sequence $\sigma$ of $j$ colors, define $A_{\sigma}^+(n,N)$ to be the
existence of $j$ disjoint paths in $(\Lambda_N\setminus\Lambda_n)\cap
\mathbb H^+$ from $\partial\Lambda_n\cap \mathbb H^+$ to
$\partial\Lambda_N\cap\mathbb H^+$, colored counterclockwise according
to $\sigma$.

\begin{theorem}\label{universal_exponent}
  For every $0<n<N$, there exist two constants $c,C\in(0,\infty)$ such that
  \begin{align*}
    &c\left(\frac nN\right)^{2}\le \P \big[A_{01001}(n,N)\big] \le C \left(\frac nN\right)^{2},\\
    &c\left(\frac nN\right)^{2}\le \P \big[A_{010}^+(n,N)\big] \le C \left(\frac nN\right)^{2},\\
    &c\frac nN~\le~\P \big[A_{01}^+(n,N)\big]~\le~C
    \frac nN.  \end{align*}
\end{theorem}
The three computations are based on the same type of ingredient, and
we refer to \cite{Wer09} for a complete derivation. An important
observation is that the proof of the above is based only on
Theorem~\ref{quasi-multiplicativity}, Proposition~\ref{landing} and
crossing estimates (Corollary~\ref{any_shape}). It does not require conformal invariance.
\begin{figure}[ht!]
  \begin{center}
    \includegraphics[width=0.50\textwidth]{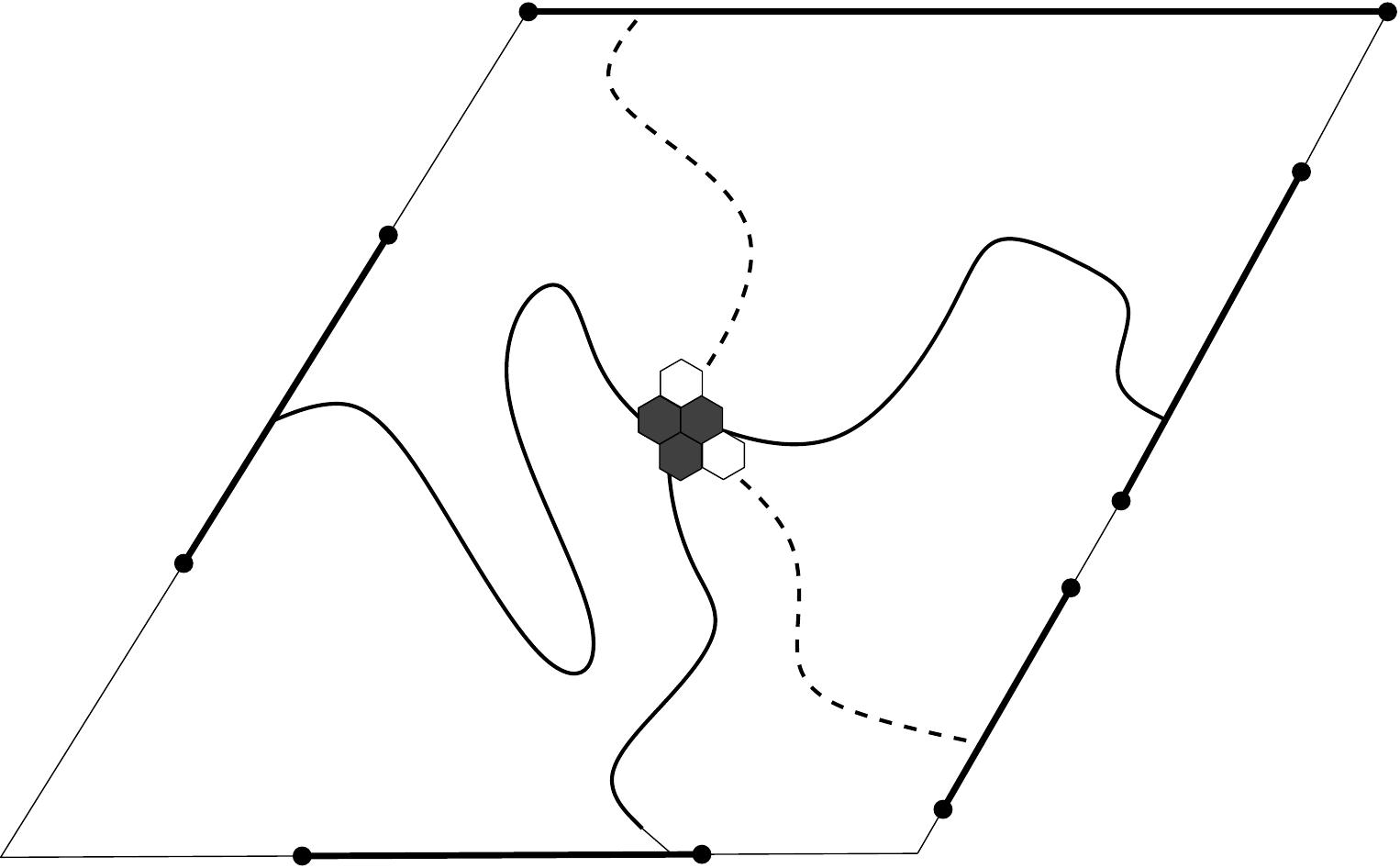}
  \end{center}
  \caption{\label{fig:five_arm} The landing sequences used to prove an upper bound on the probability of the 5-arm event. At most one site with five such arms exists.}
\end{figure}

\begin{proof}
  We only give a sketch of the proof of the first statement; the others are derived from similar arguments.

  Consider percolation in a large $N\times N$ rectangle $R$, and mark
  five boundary intervals according to Fig.~\ref{fig:five_arm}. It is
  easy to check that there is at most one site in the rectangle which
  is connected to these boundary arcs by disjoint arms of the depicted
  colors; in other words, the expected number of such hexagons is at
  most $1$. On the other hand, by arm localization, the probability
  for each of the hexagons in the middle $(N/3) \times (N/3)$
  rectangle $R'$ to exhibit $5$ such arms is given up to
  multiplicative constants by the probability of $5$ arms between
  radii $1$ and $N$, leading to the upper bound $$\P
  \big[A_{01001}(1,N)\big] \le C \left(\frac 1N\right)^{2}.$$

  To get the corresponding lower bound, we need to show that such a hexagon
  can be found with positive probability within $R'$, and this in turn is a
  consequence of crossing estimates (Corollary~\ref{any_shape}). One way to proceed is as follows. Let
  $\Gamma$ be the highest horizontal, open crossing of $R$, provided such a
  crossing exists (which occurs with positive probability by Corollary~\ref{any_shape}).
  $\Gamma$ goes through $R'$ with positive probability, and by definition, any
  hexagon on $\Gamma$ is connected to the top side of $R$ by a closed path. On
  the other hand, with positive probability, $\Gamma$ itself is connected to the
  bottom side of $R$ by an open path; let $\Gamma'$ be the right-most such path,
  and let $X$ be the hexagon at which $\Gamma$ and $\Gamma'$ intersect. Still
  with positive probability from Corollary~\ref{any_shape}, $X\in R'$; and the absence
  of an open path further to the right imposes the existence of a closed path
  below $\Gamma$, connecting a neighbor of $X$ to the right-hand edge of $R$.
  Collecting all the information given by the construction, we see that from $X$
  start five macroscopic disjoint arms of the same colors as in
  Fig.~\ref{fig:five_arm}, from which the lower bound follows: $$c\left(\frac
  1N\right)^{2}\le \P \big[A_{01001}(1,N)\big].$$

  Similar bounds for $\P \big[A_{01001}(n,N)\big]$ may then be obtained invoking quasi-multiplicativity (Theorem~\ref{quasi-multiplicativity}), thus ending the argument.
\end{proof}

With the Reimer inequality (see the introduction), the first
inequalities in the previous result imply several interesting
inequalities on arm-exponents. For instance,
$$\Pp\big[A_{\sigma 1}(n,N)\big] \le
\Pp\big[A_{\sigma}(n,N)\big]\cdot\Pp\big[A_1(n,N)\big].$$ Since
$\Pp[A_1(n,N)]~\ge~(n/N)^\alpha$ for some constant $\alpha>0$, we
deduce from the previous theorem that
\begin{align}
  \label{exponents_in}\Pp\big[A_{1010}(n,N)\big]\ge(n/N)^{2-\alpha}\text{~and~}
  \Pp\big[A_{101010}(n,N)\big]\le(n/N)^{2+\alpha}.
\end{align}
These bounds are crucial for the study of the dynamical percolation
\cite{Gar10}, the scaling relations, and for the alternative proof of convergence to SLE(6)
presented in \cite{Wer09}.

% This result has an interesting corollary:
%
% \begin{corollary}\label{6-4_arm_exponent}
%   Fix $p\in(0,1)$. There exists $\alpha>0$ such that for every
%   $0<k<n\le L_p$,
%   \begin{align*}\Pp\big[A_{ocococ}(k,n)\big] &\lesssim (k/n)^{2+\alpha}\\
%     \Pp\big[A_{ococ}(k,n)\big] &\ges (k/n)^{2-\alpha}.
%   \end{align*}
% \end{corollary}

% The first inequality is useful since it relates to convergence to
% $\SLE(6)$, as mentioned before. The second one is also interesting,
% especially when $k=1$, since it implies the existence of at least
% $n^\alpha$ \emph{pivotal points}. This fact is crucial in the study
% of the near-critical regime, as well as the dynamical percolation,
% see \cite{Gar10} and references therein.
%
% \begin{proof}
%   Note that for any sequence $\sigma$, Reimer's inequality implies
%   $$\Pp\big[A_{\sigma o}(k,n)\big] \le
%   \Pp\big[A_{\sigma}(k,n)\big]\cdot\Pp\big[A_1(k,n)\big].$$ The
%   result also holds with $\{o\}$ replace by $\{c\}$. Since
%   $\Pp[A_1(k,n)]~\ge~(k/n)^\alpha$ for some constant $\alpha$, we
%   deduce the result from the fact that
%   $\Pp\big[A_{ocooc}\big]\asymp(k/n)^2$.
% \end{proof}

\subsection{Critical arm exponents}\label{sec:exponents}

The fact that the driving process of $\SLE$ is a Brownian motion paves
the way to the use of techniques such as stochastic calculus
in order to study the properties of $\SLE$ curves. Consequently,
$\SLE$s are now fairly well understood. Path properties have been
derived in \cite{RS05}, their Hausdorff dimension can be computed
\cite{Bef04,Bef08}, etc. In addition to this, several critical
exponents can be related to properties of the interfaces, and thus be
computed using $\SLE$.

A color-switching argument very similar to the one harnessed in
Lemma~\ref{color_switching} shows that when one exponent
$\alpha_\sigma$ exists for some polychromatic sequence $\sigma$ (here
polychromatic means that the sequence contains at least one 0 and one
1), the exponents $\alpha_{\sigma'}$ exist for every polychromatic
sequence $\sigma'$ of the same length as $\sigma$, and furthermore
$\alpha_{\sigma'}=\alpha_\sigma$. From now on, we set $\alpha_j$ to be
the exponent for polychromatic sequences of length $j$. By extension,
we set $\alpha_1$ to be the exponent of the one-arm event.

\begin{theorem}[\cite{LSW02c,SW01}]
  \label{exponents}
  The exponents $\alpha_j$
  exist. Furthermore, $$\alpha_1=\frac{5}{48}\text{ and }\alpha_j =
  \frac{j^2-1}{12}\text{ for }j>1.$$
\end{theorem}

The proof of this is heavily based on the use of Schramm--Loewner
Evolutions. We sketch the proof and we refer the reader to existing
literature on the topic for details \cite{LSW02c,SW01}. The argument
is two-fold. First, arm-exponents can be related to the corresponding
exponents for $\SLE$. And second, these exponents can be computed
using stochastic and conformal invariance techniques. We will not
describe the second step, since the computation can be found in many
places in the literature already, and it would bring us far from our
main subject of interest in this review.

\begin{lemma}
  Let $\sigma$ be a polychromatic sequence of length $j$. For $R>1$,
  $\P[A_\sigma(m,Rm)]$ converges as $m$ goes to $\infty$ to a quantity
  which will be denoted by $P[A^{\SLE_6}_\sigma(1,R)]$ (see the proof
  for a description). Furthermore, $$\lim_{R\rightarrow
    \infty}\frac{\log P[A^{\SLE_6}_\sigma(1,R)]}{\log R} =
  - \begin{cases}
    \frac{j^2-1}{12}&\text{if $j>1$,}\\
    \frac5{48}&\text{if $j=1$.}
  \end{cases}$$
\end{lemma}

\begin{proof}[Proof (sketch)]
  Let us first deal with $j=1$. Let $\Lambda$ be the box centered at
  the origin with hexagonal shape and edge-length 1. Consider a
  exploration process in the discrete domain $(R\Lambda)_\delta$
  defined as follows:
  \begin{itemize}
  \item It starts from the corner $R$.
  \item Inside the domain, the exploration $\gamma$ turns left when it
    faces an open hexagon, and right otherwise.
  \item On the boundary of $(R\Lambda)_\delta\setminus \gamma$,
    $\gamma$ carries on in the connected component of
    $(R\Lambda)\setminus \gamma$ containing the origin (it always
    bumps in such a way that it can reach the origin eventually).
  \end{itemize}
  The existence of an open path from $\partial\Lambda$ to
  $\partial(R\Lambda)$ corresponds to the fact that the exploration
  does not close any counterclockwise loop before reaching $\Lambda$.

  It can be shown that the exploration $\gamma$ converges to a
  so-called radial $\SLE_6$ \cite{LSW02c}, so that the probability of
  $\smash{\P} [A_1(m,Rm)]$ converges to the probability that such a
  $\SLE_6$ does not close counterclockwise loops before reaching
  $\Lambda$ (denote this probability by $P[A^{\SLE_6}_1(1,R)]$). This
  quantity has been computed in \cite{LSW02c} and has been proved to satisfy
  $$\frac{\log P[A^{\SLE_6}_1(1,R)]}{\log R}\rightarrow
  -\frac{5}{48}\text{ as }R\text{ goes to }\infty,$$
  thus concluding the proof in this case.

  Let us now deal with $\alpha_j$ for $j>1$ even. Let us consider the
  case of the sequence of alternative colors $\sigma$ with length $j$
  (we do not loose any generality since all the polychromatic
  exponents with the same number of colors are equal). In terms of the
  exploration path, the event $A_\sigma(m,Rm)$ corresponds to the
  exploration process doing $j$ inward crossings of the annulus
  $(R\Lambda)\setminus \Lambda$. The probability of the event for
  $\SLE_6$, called $A_\sigma^{\SLE_6}(1,R)$, was also estimated in
  \cite{LSW01a,LSW01b} and has been proved to satisfy
  $$\frac{\log P[A^{\SLE_6}_\sigma(1,R)]}{\log R}\rightarrow
  - \frac {j^2-1}{12}\text{ as }R\text{ goes to }\infty,$$ thus concluding
  the proof in this case. The case of $j$ odd can also be handled similarly. Let us
  mention that the previous paragraphs constitute a sketch of proof
  only, and the actual proof is fairly more complicated, we refer to
  \cite{LSW02c,SW01} (or \cite{Wer04,Wer09}) and the references therein
  for a full proof.
\end{proof}

% \begin{lemma}
%   Let $\sigma$ be a polychromatic sequence of length $j$. Assume
%   that for any $R>0$, $\P[A_\sigma(m,Rm)]$ converges as $m$ goes to
%   $\infty$.
%
%   Furthermore, for any $\ep>0$, there exists $R>0$ such that for any
%   $N$ large enough,
%   $$\left|\frac{\log \P[A_\sigma(N)]}{\log
%     N}-\lim_{m\rightarrow\infty}\frac{\log \P[A_\sigma(m,Rm)]}{\log
%     R}\right|\le \ep.$$
% \end{lemma}

We are now in a position to prove Theorem~\ref{exponents}.

\begin{proof}[Proof of Theorem~\ref{exponents}]
  % The previous lemma implies that in order to compute arm-exponents,
  % it suffices to relate discrete probability that
  % $\lim_{m\rightarrow\infty}\P[A_j(m,Rm)]$ exists and to compute its
  % limit as $R$ goes to $\infty$.

  Fix $\ep>0$. It is sufficient to study the convergence along
  integers of the form $R^n$ since Corollary~\ref{any_shape} enables one to relate the
  probabilities of $A_\sigma(1,R^n)$ to the one of $A_\sigma(1,N)$ for
  any $R^n\le N<R^{n+1}$. Using Theorem~\ref{quasi-multiplicativity}
  iteratively, there exists a universal constant $C>1$ independent of $R$ such that for
  any $n$,
  \begin{equation}\label{qu}
    \left|\log \P[A_\sigma(R^n)]-\sum_{k=0}^{n-1}\log
      \P[A_\sigma(R^k,R^{k+1})]\right|\le n\log C.
  \end{equation}
  The previous lemma implies that $\P[A_\sigma(m,Rm)]$ converges as $m$
  goes to $\infty$. Therefore,
  $$\frac{1}{n}\sum_{k=0}^{n-1}\log
  \P[A_\sigma(R^k,R^{k+1})]\longrightarrow \log
  P[A^{\SLE_6}_\sigma(1,R)]\text{ as }n\text{~goes~to~}\infty.$$ Now,
  let $R$ be large enough that $\log C/\log R\le \ep/2$. The statement
  follows readily by dividing \eqref{qu} by $n\log R$ and plugging the
  previous limit into it.
\end{proof}

\subsection{Fractal properties of critical percolation}

Arm exponents can be used to measure the Hausdorff dimension of sets
describing critical percolation clusters. A set $S$ of vertices of the
triangular lattice is said to have dimension $d_S$ if the density of
points in $S$ within a box of size $n$ behaves as $n^{-x_S}$, with $x_S
= 2 - d_S$.  The codimension $x_S$ is related to arm exponents in many
cases:
\begin{itemize}
\item The $1$-arm exponent is related to the existence of long
  connections, from the center of a box to its boundary. It measures
  the Hausdorff dimension of big clusters, like the incipient infinite
  cluster (IIC) as defined by Kesten \cite{Kes86}. For instance, the
  IIC has a Hausdorff dimension equal to $2-5/48 =91/48$.
\item The monochromatic $2$-arm exponent describes the size of the
  \emph{backbone} of a cluster. It can be shown using the BK
  inequality that this exponent is strictly smaller than the one-arm
  exponent, hence implying that this backbone is much thinner than the
  cluster itself. This fact was used by Kesten \cite{Kes86} to prove
  that the random walk on the IIC is sub-diffusive (while it has been
  proved to converge toward a Brownian Motion on a supercritical
  infinite cluster, see \cite{BB07,MP07} for instance).
\item The polychromatic $2$-arm exponent is related to the boundary
  points of big clusters, which are thus of fractal dimension
  $7/4$. This exponent can be observed experimentally on interfaces
  (see \cite{DSB03,SRG85} for instance).
\item The 4-arm exponent with alternating colors counts the pivotal
  sites (see the next section for more information). The dimension of
  the set of pivotal sites is thus $3/4$. This exponent is crucial in
  the study of noise-sensitivity of percolation (see
  \cite{SS10,GPS10b} and references therein).
\end{itemize}

\section{The critical point of percolation and the near-critical regime}

We now move away from the critical regime and start to study site percolation with arbitrary $p$ (keeping in mind that we are mostly interested in the study of $p$ near $p_c=1/2$).
\subsection{Proof of Theorem~\ref{thm:Kesten}}

We are arriving at a milestone of modern probability, Kesten's
``$p_c=1/2$'' theorem (Theorem~\ref{thm:Kesten}). Originally, the
statement was proved in the case of bond percolation on the square
lattice, but the same arguments apply to site percolation on the
triangular lattice. Besides, the method we present here is not the
historical one, and was introduced by Bollob\'as and Riordan
\cite{BR06b}. Observe that Corollary~\ref{lower_bound_critical}
implies that $p_c\ge \frac12$; therefore, we only need to prove that
$p_c\le \frac 12$ to show Theorem~\ref{thm:Kesten} and we focus on
this assertion from now on.

From now on, $[0,n]\times[0,m]$ denotes the set of points of the form
$k\cdot 1+\ell\cdot e^{i\pi/3}$, with $0\le k\le n$ and $0\le \ell\le
m$.  Let us start by the following proposition asserting that if some
crossing probability is too small, then the probability of the origin
being connected to distance $n$ decays exponentially fast in $n$.

\begin{proposition}\label{exponential_decay}
  Fix $p\in (0,1)$ and assume there exists $L\in\mathbb N$ such that
  $$\Pp([0,L]\times[0,2L]\text{ is crossed horizontally})<
  \frac{1}{36 e}.$$
  Then for any $n\ge L$, $\Pp(0\leftrightarrow \partial
  \Lambda_n)\le 6\exp[-n/(2L)]$.
\end{proposition}

The previous proposition, in conjunction with the fact that, for
$p>1/2$, the probability of having a closed crossing of
$[0,2n]\times[0,n]$ tends to zero as $n$ tends to infinity (this is
non-trivial and will be proved later), implies that $p_c\le \frac12$.
Indeed, assume that these probabilities tend to 0 for $p>1/2$, the
probability that there exists a closed circuit of length $n$
surrounding the origin is thus smaller than $(2n)^2\cdot 6e^{-n/(2L)}$
using the previous proposition for closed sites instead of open
ones. The Borel-Cantelli Lemma implies that there exists almost surely
only a finite number of closed circuits surrounding the origin. As a
consequence, there exists an infinite open cluster almost surely. If
this is true for any $p>1/2$, it means that $p_c\le \frac12$.

\begin{proof}
  Let $m>0$ and consider the rectangles
  \medbreak
  \begin{tabular}{ll}
    $R_1=[0,m]\times[0,2m],$ & $R_2=[0,m]\times[m,3m], $\\
    $R_3=[0,m]\times[2m,4m],$ & $R_4=[m,2m]\times[0,2m],$\\
    $R_5=[m,2m]\times[m,3m],$ & $R_6=[m,2m]\times[2m,4m],$\\
    $R_7=[0,2m]\times[m,2m],$ & $R_8=[0,2m]\times[2m,3m]. $
  \end{tabular}

  \medbreak

  These rectangles have the property that whenever
  $[0,2m]\times[0,4m]$ is crossed horizontally, two of the rectangles
  $R_i$ (possibly the same) are crossed in the short direction by \emph{disjoint}
  paths. We deduce, using the BK inequality, that
  \begin{align*}
    \Pp\big([0,2m]\times[0,4m]&\text{~is crossed horizontally}\big)\\
    & \le
    36~\Pp\big([0,m]\times[0,2m]\text{~is crossed\ horizontally}\big)^2.
  \end{align*}
  Iterating the construction, we easily obtain that for every $k\ge
  0$,
  \begin{align*}
    36\Pp\big([0,2^km]\times[0,&2^{k+1}m]\text{~is
      crossed\ horizontally}\big)\\
    &\ \ \, \le \left(36\Pp\big([0,m]\times[0,2m]\text{~is
        crossed\ horizontally}\big)\right)^{2^k}.
  \end{align*}
  In particular, if $m=L$ and $36\Pp\big([0,n]\times[0,2n]\text{~is cros.\ hor.}\big)<1/e$, we
  deduce for $n=2^k L$:
  $$\Pp(0\leftrightarrow \partial \Lambda_n)~\le~6\Pp\big([0,n]\times[0,2n]\text{~is cros.\ hor.}\big)\le 6e^{-n/L}.$$
  We used the fact that at least one out of six rectangles with dimensions $n\times2n$ must be crossed in order for the origin to be connected to distance $n$. The claim follows for every $n$ by monotonicity.
\end{proof}

% The proposition can be reformulated in the following way. Let
% $\mathcal C_{\rm sub}$ denote the statement that there exists $L>0$
% such that
% $$\Pp([0,L]\times[0,2L]\text{ is crossed horizontally})<
% \frac{1}{e36}.$$ The condition $\mathcal C_{\rm sub}$ on $p$
% is a criterion of sub-criticality, and the condition that it holds
% for $1-p$ is a criterion of super-criticality.
%
% \bigskip
We now need to prove the following non-trivial lemma.

\begin{lemma}\label{lem:long_path}
  Let $p<1/2$, there exist
  $\ep=\ep(p)>0$ and $c=c(p)>0$ such that for every $n\ge 1$,
  \begin{equation}\label{long_path}
    \Pp\big([0,n]\times[0,2
    n]\text{~is crossed horizontally}\big) \le cn^{-\ep}.
  \end{equation}
\end{lemma}

In order to prove this lemma, we consider a more general question. We
aim at understanding the behavior of the function $p\mapsto \Pp(A)$
for a non-trivial increasing event $A$ that depends on the states of a
finite set of sites (think of this event as being a crossing
event). This increasing function is equal to $0$ at $p=0$ and to $1$
at $p=1$. We are interested in the range of $p$ for which its value is
between $\ep$ and $1-\ep$ for some predetermined positive $\ep$ (this
range is usually referred to as a \emph{window}). Under certain
conditions on $A$, the window will become narrower when $A$ depends on
a larger number of sites. Its width can be bounded above in terms of
the size of the underlying graph. This kind of result is known as
\emph{sharp threshold} behavior.

The study of $p\mapsto \Pp(A)$ harnesses a differential equality known
as Russo's formula:

\begin{proposition}[Russo~\cite{Rus78}, Section~2.3 of~\cite{Gri99}]
  \label{russo_influence}Let $p\in(0,1)$ and $A$ an increasing event
  depending on a \emph{finite} set of sites $V$. We
  have $$\frac{d}{dp} \Pp(A) = \sum_{v\in V}\Pp(v\text{ pivotal for
  }A),$$ where $v$ is pivotal for $A$ if $A$ occurs when $v$ is open,
  and does not if $v$ is closed.
\end{proposition}

If the typical number of pivotal sites is sufficiently large whenever
the probability of $A$ is away from $0$ and 1, then the window is
necessarily narrow. Therefore, we aim to bound from below the expected
number of pivotal sites.  We present one of the most striking results
in that direction.

\begin{theorem}[Bourgain, Kahn, Kalai, Katznelson, Linial \cite{BKKKL92}, see also
  \cite{KKL88,Fri04,FK96,KS06}]\label{maximum_influence}
  Let $p_0>0$. There exists a constant $c = c(p_0) \in (0,\infty)$ such
  that the following holds. Consider a percolation model on a graph $G$
  with $|V|$ denoting the number of sites of $G$.  For every
  $p\in[p_0,1-p_0]$ and every increasing event $A$, there exists $v\in
  V$ such that \[ \Pp(v\text{ pivotal for }A)~\geq c
  \,\Pp(A)\big(1-\Pp(A)\big)\frac{\log |V|}{|V|}.  \]
\end{theorem}

This theorem does not imply that there are always many pivotal sites,
since it deals only with the maximal probability over all sites. It
could be that this maximum is attained only at one site, for instance
for the event that a particular site is open. There is a particularly
efficient way (see \cite{BR06a, BR06b,BD10}) to avoid this problem.
In the case of a \emph{translation-invariant event} $A$ on a torus
with $n$ vertices, sites play a symmetric role, so that the
probability to be pivotal is the same for all of
them. Proposition~\ref{russo_influence} together with
Theorem~\ref{maximum_influence} thus imply that in this case, for
$p\in [p_0,1-p_0]$,
\[ \frac{d}{dp}\Pp(A)\geq c \big(\Pp(A)(1-\Pp(A)\big)\log n.  \]
Integrating the previous inequality between two parameters
$p_0<p_1<p_2<1-p_0$, we obtain \[ \frac{\mathbb P_{p_1}(A)}{1-\mathbb
  P_{p_1}(A)}\leq \frac{\mathbb P_{p_2}(A)}{1-\mathbb
  P_{p_2}(A)}n^{-c(p_2-p_1)}.\] If we further assume that $\mathbb
P_{p_2}(A)\le r<1$, there exist $c,C>0$ depending on $r$ only such that
\begin{equation}\label{important}
  \mathbb P_{p_1}(A)\leq Cn^{-c(p_2-p_1)}.
\end{equation}

%Now that the theory is settled, we can prove the fundamental lemma
%which shows that, when $p<\frac12$, $\mathbb
%P_p([0,n]\times[0,2n]\text{ is crossed horizontally})$ tends to zero
%as $n$ tends to infinity.

% \begin{figure}[ht!]
%   \begin{center}
%     \includegraphics[width=0.50\textwidth]{translation}
%   \end{center}
%   \caption{\label{fig:translation} The rectangles $R_1$, \ldots,
%   $R_{16}$. They are all translates of $R_1$.}
% \end{figure}

We are now in a position to prove Lemma~\ref{lem:long_path}. The proof uses Theorem~\ref{maximum_influence}. We consider a
carefully chosen translation-invariant event for which we can prove
sharp threshold. Then, we will bootstrap the result to our original
event. Let us mention that Kesten proved a
sharp-threshold in \cite{Kes80} using different arguments. Several other approaches
have been developed, see Theorem 3.3 and Corollary 2.2 of \cite{Gri10}
for instance.

\begin{proof}
  % Consider the torus $\mathbb T_{2n}$ of size $2n$, which can be
  % seen as a quotient of $[0,2n]^2$. Let $B_n$ be the event
  % that there exists a horizontal crossing of a rectangle with
  % dimensions $(n/2,2n)$ in $\mathbb T_{2n}$. This event is invariant
  % under translations, and $$ \mathbb P_p(B_n)\le \mathbb
  % P_p([0,n]\times[0,2 n]\text{~is crossed horizontally}).$$ Using
  % \eqref{important}, we see that it is sufficient to prove that
  % $\P(B_n)$ stays bounded away from 1 as $n\rightarrow\infty$. If
  % every rectangle of the form $[kn,(k+1)\frac n2]\times[\ell n,
  % (\ell+2) n]$ for $k\in \{0,\ldots,3\}$ and $\ell\in \{0,1\}$ is
  % crossed vertically by a closed path, then $B_n$ does not occur. By
  % Corollary~\ref{any_shape}, each of these rectangles has
  % probability larger than some universal constant $c>0$ of being
  % crossed vertically by a closed path. Using the Harris inequality,
  % we obtain that $\P(B_n)\le 1-c^8$, which concludes the proof.

  Consider the torus $\mathbb T_{4n}$ of size $4n$ (meaning $[0,4n]^2$ with sites $(0,k)$ identified with $(4n,k)$, for all $0\le k\le 4n$ and $(\ell,0)$ identified with $(\ell,4n)$, for all $0\le \ell\le 4n$). Let $B$ be the
  event that there exists a vertical closed path crossing some rectangle
  with dimensions $(n/2,4n)$ in $\mathbb T_{4n}$. This event is
  invariant under translations and satisfies
  \begin{equation*}
    \P(B)\geq \P\big([0,n/2]\times[0,4
    n]\text{~is crossed by a closed path vertically}\big) \geq c>0
  \end{equation*}
  uniformly in $n$. Since $B$ is decreasing, we can apply
  \eqref{important} with $B^c$ to deduce that for $p<1/2$, there exist
  $\ep,c>0$ such that
  \begin{equation}\label{aaa}
    \Pp(B)\geq 1-cn^{-\ep}.
  \end{equation}
  If $B$ holds, one of the $16$ rectangles of the form $[k\frac
  n2,(k+2)\frac n2]\times[\ell n, (\ell+2) n]$ for $k\in
  \{0,\ldots,7\}$ and $\ell\in \{0,1\}$ is crossed vertically by a
  closed path. We denote these events by $A_1$, \ldots, $A_{16}$ ---
  they are translates of the event that $[0,n]\times[0,2 n]$ is
  crossed vertically by a closed path. Using the Harris inequality in
  the second line, we find
  \begin{align*}
    \Pp(B) &= 1-\Pp(B^c)=1-\Pp\left(\bigcap_{i=1}^{16}A_i^c\right)\leq 1-\prod_{i=1}^{16}\Pp(A_i^c)\\
    &=1-\left[1-\Pp\big([0,n]\times[0,2 n]\text{~is crossed vertically
        by a closed path}\big)\right]^{16}.
  \end{align*}
  Plugging \eqref{aaa} into the previous inequality, we deduce
  \begin{equation*}
    \Pp\left([0,n]\times[0,2
      n]\text{~is crossed vertically by a closed path}\right) \geq
    1-(cn^{-\ep})^{1/16}.
  \end{equation*}
  Taking the complementary event, we obtain the claim. This
  application of the Harris inequality is colloquially known as the
  \emph{square-root trick}.
\end{proof}

\subsection{Definition of the correlation length}

We have studied how probabilities of increasing events evolve as
functions of $p$. If $p$ is fixed and we consider larger and larger
rectangles (of size $n$), crossing probabilities go to 1 whenever
$p>1/2$, or equivalently to 0 whenever $p<1/2$. But what happens if
$(p,n)\rightarrow (1/2,\infty)$ (this regime is called the
\emph{near-critical regime})?

If one looks at two percolation pictures in boxes of size $N$, one at
$p>0.5$, and one at $p<0.5$, it is only possible to identify which is supercritical and which is subcritical when $N$ is large enough. The scale at
which one starts to see that $p$ is not critical is called the
\emph{correlation length}. Interestingly, it can naturally be expressed in
terms of crossing probabilities.

\begin{definition}\label{def:correlation_length}
  For $\ep>0$ and $p\le1/2$, define the correlation length as
  $$
  L_p(\ep):=\inf\big\{n>0:\Pp\big([0,n]^2\text{is crossed horizontally
    by an open path}\big)\le \ep\big\}.
  $$
  Extend the definition of the correlation length to every $p\ge1/2$
  by setting $L_p(\ep):=L_{1-p}(\ep)$.
\end{definition}

Note that the fact that $L_p(\ep)$ is finite for $p\ne 1/2$ comes from the
fact that crossing probabilities converge to $0$ when $p<1/2$.

Let us also mention that taking $[0,n]^2$ in the definition of the
correlation length is not crucial. Indeed, the following result,
called a Russo-Seymour-Welsh result, implies that one could
equivalently define the correlation length with rectangles of other
aspect ratios, and that it would only change the value of the
corresponding $\ep$.
% Recall that $[0,n]\times[0,m]$ denotes the set of points of the form
% $k\cdot 1+\ell\cdot e^{i\pi/3}$, with $0\le k\le n$ and $0\le
% \ell\le m$.
We state the following theorem without proof.

\begin{theorem}[see \emph{e.g.} \cite{Gri99, Kes82}]\label{RSWstrong}
  Let $p_0>0$. There exists a strictly increasing continuous function
  $\rho_{p_0}:[0,1]\rightarrow[0,1]$ such that $\rho_{p_0}(0)=0$
  satisfying the following property: for every $p\in (p_0,1-p_0)$ and
  every $n>0$,
  $$\rho_{p_0}(\delta)\le\Pp([0,2n]\times[0,n]\text{~is crossed
    horizontally by an open path})\le 1-\rho_{p_0}(\delta),$$ where
  $$\delta:=\Pp([0,n]^2\text{~is crossed horizontally by an open path}).$$
\end{theorem}

From now on, fix $p_0\in(0,\frac12)$. Let us mention that $\ep$ is
chosen in the following fashion. We want to argue that percolation in
boxes of size $n\gg L_p(\ep)$ looks subcritical or supercritical
depending on $p<1/2$ or $p>1/2$. To do so, we would like to have that
$$\Pp([0,2n]\times[0,n]\text{ crossed vertically})<\frac{1}{\binom 82e}$$ for
$n\ge L_p(\ep)$ and $p<1/2$. Therefore, fix $\ep=\ep(p_0)$ small enough so that
$\rho(\ep)<1/(\binom 82e)$. Keep in mind that all constants henceforth
depend on $p_0$ and $\ep>0$.

Note that with this value of $\ep$, the correlation length at
criticality equals infinity, since probabilities to be connected at
distance $n$ do not decay exponentially for $p=1/2$.

One very important feature of the correlation length is the following
property. Fix $p_0>0$ and $\ep=\ep(p_0)$. For any topological
rectangle $(\Omega,A,B,C,D)$, there exists $c>0$ such that for
$p\in(p_0,1-p_0)$ and $n<L_p(\ep)$,
\begin{equation}\label{crossing_correlation_length}
  \Pp\big[\mathcal C_{1/n}(\Omega,A,B,C,D)\big]~\ge~c.
\end{equation}
In this sense, the configuration looks critical ``uniformly in $n<L_p(\ep)$''. We do not prove this fact, which uses a variant of the RSW theory and can be found in the literature. We will see in the
next section that fractal properties below the
correlation length are also similar to fractal properties at criticality.

We conclude this section by mentioning that the correlation
length is classically defined as the ``inverse rate'' of exponential decay of the
two-point function. More precisely, since the quantity
$\Pp(0\leftrightarrow nx)$ is super-multiplicative, the quantity
$\lim_{n\rightarrow \infty}-\frac1n\log \Pp(0\leftrightarrow nx)$ is
well-defined. The correlation length is then defined as the inverse of this quantity. It is possible to prove that $L_p(\ep)$ is, up to
universal constant depending only on $p_0$ and $\ep$, asymptotically equivalent to
$\lim_{n\rightarrow \infty}-\frac1n\log \Pp(0\leftrightarrow nx)$ when
$p<1/2$ (note that Proposition~\ref{exponential_decay} gives one
inequality, see \emph{e.g.} Theorem 3.1 of \cite{Nol08} for the other
bound).

\subsection{Percolation below the correlation length}

Proposition~\ref{exponential_decay} together with the definition of
the correlation length study percolation in boxes of size $n\gg
L_p(\ep)$. The goal of this section
is to describe percolation below the correlation length. In
particular, we aim to prove that connectivity properties are
essentially the same as at criticality by proving that the variation
of $\Pp[A_\sigma(n,N)]$ as a function of $p$ is not
large provided that $N<L_p(\ep)$. As a direct consequence, $\Pp[A_\sigma(n,N)]$ remains basically the same when varying $p$ in the regime $N<L_p(\ep)$. This
fact justifies the following motto: \emph{below $L_p$, percolation
  looks critical}.

Before stating the main result, recall that
\eqref{crossing_correlation_length} easily implies the following
collection of results (they correspond to
Theorems~\ref{quasi-multiplicativity},~\ref{universal_exponent} and
Proposition~\ref{landing}), since they are consequences of crossing
estimates only.

Fix $p_0\in(0,1)$, $\ep>0$ small enough and $\delta>0$. Consider a
sequence $\sigma$. There exist $c,C\in(0,\infty)$ such that for any
$p\in(p_0,1-p_0)$ and $n_1<n_2<n_3\le L_p(\ep)$,
\small \begin{equation}\label{quasimultiplicativity_off}c \mathbb P_p
  \big[ A_{\sigma} (n_1,n_2) \big]\mathbb
  P_p\big[A_{\sigma}(n_2,n_3)\big]\le
  \P\big[A_{\sigma}(n_1,n_3)\big] \le \P
  \big[ A_{\sigma} (n_1,n_2) \big]\mathbb
  P_p\big[A_{\sigma}(n_2,n_3)\big].\end{equation}\normalsize
Furthermore, for any choice of $\delta$-well separated landing
sequences $I,J$ and for any $2n\le N\le L_p(\ep)$,
\begin{equation}\label{localite_off}\mathbb P_p\big[A^{I,J}_\sigma(n,N)\big]\le\mathbb P_p\big[A_\sigma(n,N)\big]\le C\mathbb P_p\big[A^{I,J}_\sigma(n,N)\big].\end{equation}
Finally, for every $0<n<N\le L_p(\ep)$,
\begin{align}
  &c\left(\frac nN\right)^{2}\le \mathbb P_p\big[A_{01001}(n,N)\big] \le C \left(\frac nN\right)^{2},\\
  &c\left(\frac nN\right)^{2}\le \mathbb P_p\big[A_{010}^+(n,N)\big] \le C \left(\frac nN\right)^{2},\label{3_arm_off}\\
  &c\frac nN~\le~\mathbb P_p\big[A_{01}^+(n,N)\big]~\le~C \frac nN,\\
  &\mathbb P_p\big[A_{1010}(n,N)\big]\ge
  \left(\frac{n}{N}\right)^{2-c}.\label{exponents_off}\end{align}

The last inequality comes from \eqref{exponents_in}. In words, the
quasi-multiplicativity, the fact that prescribing landing sequence
affects the probability of an arm-event by a multiplicative constant only, and the
universal exponents are still valid for $p\ne \frac12$ {\em as long as we
consider scales less than $L_p(\ep)$}. Note that the universal
constants $c$ and $C$ depend on $\sigma$, $\delta$, $p_0$ et $\ep$
only (in particular it does not depend on $p\in[p_0,1-p_0]$).

In fact, the description of percolation below the scale $L_p(\ep)$ is
much more precise: \emph{no arm-event} varies in this regime and we
get the following spectacular result.

\begin{theorem}[Kesten \cite{Kes87}]\label{no_variation} Fix a
  sequence $\sigma$ of colors,
  $p_0\in(0,1)$ and $\ep>0$ small enough. There exist
  $c,C\in(0,\infty)$ such that
  $$c\P(A_\sigma(n))\le\Pp(A_\sigma(n))\le C\P(A_\sigma(n))$$ for every $p\in(p_0, 1-p_0)$ and $n\le L_p(\ep)$.
\end{theorem}

The idea of the proof is to estimate the logarithmic derivative of
arm-event probabilities in terms of the derivative of crossing
probabilities. In order to do so, we relate the probability to be
pivotal for arm-events with the probability to be pivotal for crossing
events.

\begin{proof}[Proof of Theorem~\ref{no_variation}]
  In the proof, $C_1,C_2,\dots$ are constants in $(0,\infty)$
  depending on $p_0$ and $\ep$ only. We first treat the case of
  $\Pp[A_1(n)]$ when $p>1/2$.

  Recall that $n$ is assumed to be smaller than $L_p(\ep)$, so that
  \eqref{quasimultiplicativity_off}, \eqref{localite_off},
  \eqref{3_arm_off} and \eqref{exponents_off} are satisfied for any scale smaller than $n$.

  % We will be using the fact that crossing probabilities are bounded
  % away from 0 and 1 uniformly in $n\le L_p$.

  \begin{figure}[ht!]
    \begin{center}
      \includegraphics[width=0.70\textwidth]{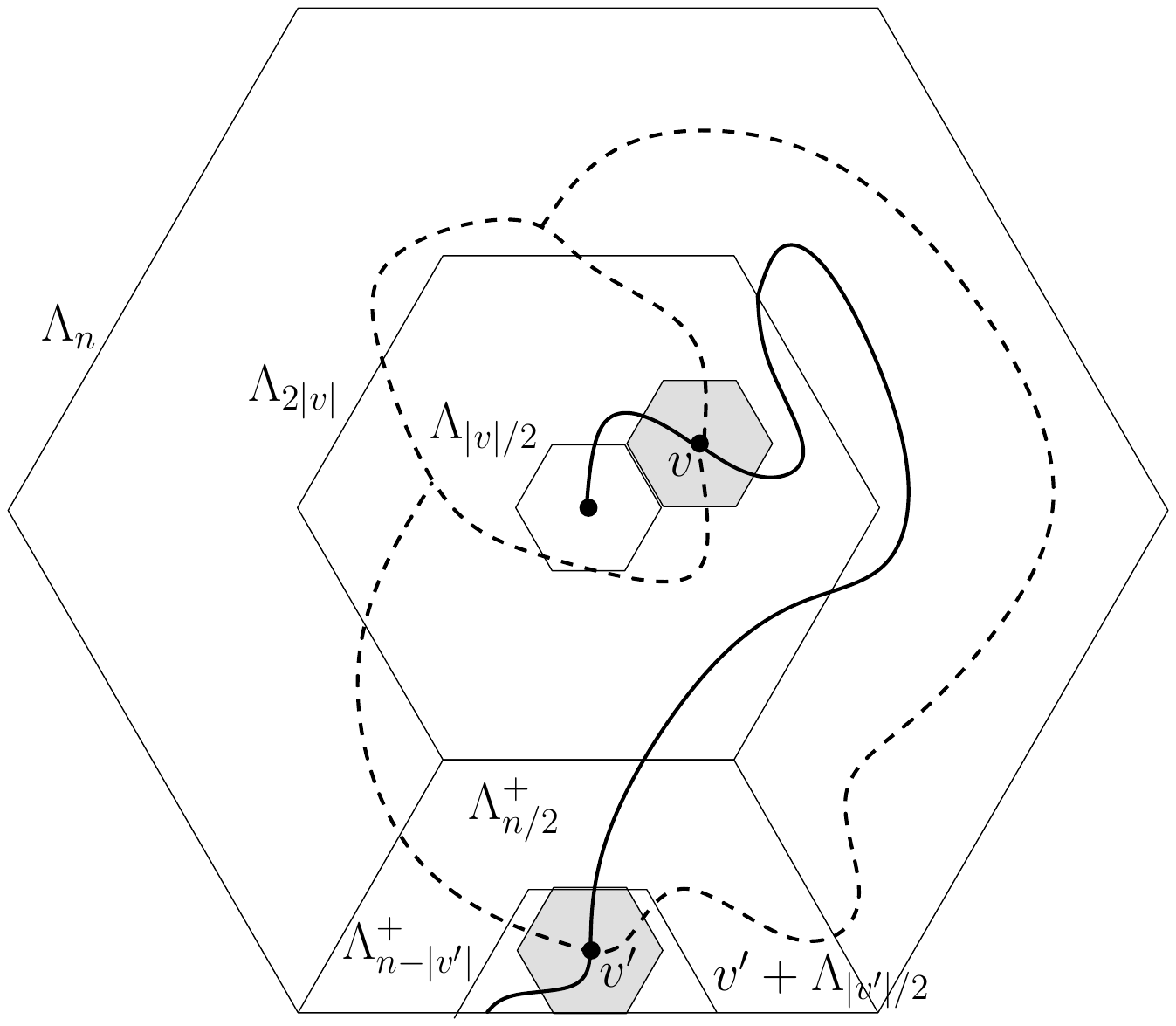}
    \end{center}
    \caption{\label{fig:pivotal}The event that $v$ and $v'$ are pivotal for
      $A_1(n)$. The dotted line corresponds to a closed circuit.}
  \end{figure}

  Russo's formula implies
  \begin{equation}\label{abcd}
    \frac{d}{dp} \Pp\big[A_1(n)\big]~=~\sum_{v\in
      \Lambda_n}\Pp\big[v\text{~pivotal for~}A_1(n)\big].
  \end{equation}
  The site $v$ is pivotal for $A_1(n)$ if and only if there are four
  arms of alternating colors emanating from it, one of the open arms
  going to the origin, the other to the boundary of the box, and the
  two closed arms together with the site $v$ forming a circuit around the origin (see
  Fig~\ref{fig:pivotal}). Let us treat two cases:
  \begin{itemize}
  \item If $|v|\le n/2$, where $|v|$ is the graph distance to the
    origin, the pivotality of the site $v$ implies that the following
    events hold: $A_1(|v|/2)$, $A_1(2|v|,n)$ and the translation of
    $A_{1010}(|v|/2)$ by $v$ (see Fig~\ref{fig:pivotal} again). We
    deduce, using independence, that
    \begin{align*}
      \Pp\big[v\text{~pivotal for~}A_1(n)\big] &\le \Pp\big[A_1(|v|/2)
      \big]\Pp\big[A_1(2|v|,n)\big]\Pp\big[A_{1010}(|v|/2)\big] \\
      &\le C_1 \Pp\big[A_1(n)\big]\Pp\big[A_{1010}(|v|/2)\big],
    \end{align*}
    where in the second line we used \eqref{quasimultiplicativity_off}
    twice together with the fact that $\Pp\big[A_1(|v|/2,2|v|)\big]\ge C_0$
 (the latter comes from the crossing estimates
    \eqref{crossing_correlation_length}). Equations
    \eqref{quasimultiplicativity_off} and \eqref{exponents_off} give
    $$\mathbb P_p\big[A_{1010}(|v|/2)\big] \le
    C_2\left(\frac{2n}{|v|}\right)^{2-c}\mathbb
    P_p\big[A_{1010}(n)\big],$$ which, when tuned into the previous
    displayed inequality, leads to
    \begin{align}\label{Case_1}
      \Pp\big[v\text{~pivotal for~}A_1(n)\big] \le
      C_3\left(\frac{2n}{|v|}\right)^{2-c}
      \Pp\big[A_1(n)\big]\Pp\big[A_{1010}(n)\big].
    \end{align}
  \item If $|v|\ge n/2$, the site $v$ is pivotal if the following
    events hold: $A_1(n/2)$, the translation of $A_{1010}(n-|v|)$ by $v$
    and the translation of $A^+_{010}(n-|v|,n/2)$ by $v$ (see
    Fig~\ref{fig:pivotal} again). We deduce, using independence, that
    \begin{align}
      \Pp\big[v&\text{~pivotal for~}A_1(n)\big] \nonumber\\
      &\le \Pp\big[A_1(n/2)
      \big]\Pp\big[A_{1010}(n-|v|)\big]\Pp\big[A^+_{010}(n-|v|,n/2)\big] \nonumber\\
      &\le C_4\Pp\big[A_1(n)
      \big]\left(n-|v|\right)^{2-c}\Pp\big[A_{1010}(n)\big]\left(\frac{2(n-|v|)}{n}\right)^2,
      \label{Case_2}
    \end{align}
    where an argument similar to the previous case was used once again
    to relate $\Pp[A_1(n/2) ]$ to $\Pp\big[A_1(n)]$ and
    $\Pp[A_{1010}(n-|v|)]$ to $\Pp[A_{1010}(n)]$. The bound
    \eqref{3_arm_off} was used to bound $\Pp[A^+_{010}(n-|v|,n/2)]$.
  \end{itemize}
  Plugging the bounds \eqref{Case_1} and \eqref{Case_2} into
  \eqref{abcd}, we easily find that for $n\le L_p(\ep)$,
  \begin{equation}\label{dif_inq}
    \frac{d}{dp}\Pp\big[A_1(n)\big]\le C_5\Pp\big[A_1(n)\big]\cdot
    n^2\Pp\big[A_{1010}(n)\big].
  \end{equation}
  We now relate $n^2\Pp\big[A_{1010}(n)\big]$ to the derivative of the
  probability of the event that $[0,n]^2$ is crossed horizontally by
  an open path. Denote this event by $E(n)$. We have
  \begin{equation*}
    \frac{d}{dp} \Pp\big[E(n)\big]~=~\sum_{v\in
      \Lambda_n}\Pp\big[v\text{~pivotal for~}E(n)\big].
  \end{equation*}
  The site $v$ is pivotal for $E(n)$ if and only if there are four
  arms of alternating colors emanating from it, the open arms going to
  the left and the right of the box, and the closed ones to the top
  and the bottom. Using \eqref{localite_off}, we find that the
  probability of $v\in[\frac n3,\frac{2n}3]^2$ being pivotal for
  $E(n)$ is larger than a universal constant times the probability of
  having four arms of alternating colors going from $v$ to the
  boundary of $[0,n]^2$. Since $[0,n]^2\subset (v+\Lambda_{2n})$, this
  implies immediately that
  $$\Pp\big[v\text{~pivotal for~}E(n)\big]\ge \frac{1}{C_6}\Pp\big[A_{1010}(2n)\big]\ge \frac{1}{C_7}\Pp\big[A_{1010}(n)\big].$$
  Once again, quasi-multiplicativity was used in a crucial way in order to obtain the last inequality. By summing on vertices $v$ in $[\frac n3,\frac{2n}3]^2$, we get
  \begin{equation}\label{abcd2}
    \frac{d}{dp} \Pp\big[E(n)\big]~\ge~\frac{1}{9C_7}n^2\Pp\big[A_{1010}(n)\big].
  \end{equation}
  Altogether, we find that for $n\le L_p(\ep)$,
  $$\frac{d}{dp}\Pp\big[A_1(n)\big]\le C_5\Pp\big[A_1(n)\big]\cdot
  n^2\Pp\big[A_{1010}(n)\big]\le
  9C_5C_7\Pp\big[A_1(n)\big]\frac{d}{dp}\Pp\big[E(n)\big].$$ Since for
  $p'\in(\frac12,p)$, $L_{p'}(\ep)\ge L_p(\ep)$, we deduce that
  \begin{align*}\label{dif}\log \Pp\big[A_1(n)\big]-\log
    \P\big[A_1(n)\big]~&\le~C_8\int_{1/2}^p\frac{d}{dp'}\mathbb P_{p'}\big[E(n)\big]~dp'\\
    &=C_8( \Pp[E(n)]-\P[E(n)])\le C_8,\end{align*} which is the claim.

  % It remains to prove that the right-hand side is of order $1$.
  % Theorem~\ref{quasi-multiplicativity} and
  % Corollary~\ref{6-4_arm_exponent} imply
  % $$\mathbb P_{p'}\big[A_{ococ}(k)\big] \lesssim (n/k)^{2-\alpha}~\mathbb
  % P_{p'}\big[A_{ococ}(n)\big],$$ for $n\le L_p\le L_{p'}$ (since
  % $1/2<p'<p$). Put into \eqref{dif}, it gives
  % \begin{align*}
  %   \log\Pp\big[A_1(n)\big]-\log \P\big[A_1(n)\big] &\lesssim
  %   \int_{1/2}^p~\left(\sum_{v\in
  %     \Lambda_n}\left(2n/|v|\right)^{2-\alpha}\mathbb
  %     P_{p'}\big[A_{ococ}(n)\big]\right)~dp'\\
  %   &\lesssim \int_{1/2}^p n^2\mathbb
  %   P_{p'}\big[A_{ococ}(n)\big]~dp'.
  % \end{align*}
  % To conclude, Russo's formula implies
  % \begin{align*}
  %   1&\ge\Pp([0,n/2]^2\text{~is crossed})-\P([0,n/2]^2\text{~is
  %   crossed}) \\
  %   &=\int_{1/2}^p \sum_{v\in \Lambda_{n/2}} \mathbb P_{p'}\big[v
  %   \text{~pivotal for }[0,n/2]^2\text{~being crossed}\big]~dp'\\
  %   &\ge \int_{1/2}^p \frac{3n^2}4 \mathbb
  %   P_{p'}\big[A_{ococ}(n)\big]~dp',\end{align*} where we have used
  % the fact that $v$ is pivotal for the event $\{[0,n/2]^2\text{ is
  % crossed}\}$ if there are four arms of alternating colors going to
  % the boundary of $[0,n/2]^2$. In particular, we find the required
  % bound
  % \begin{align*}
  %   \log\Pp\big[A_1(n)\big]-\log \P\big[A_1(n)\big]~&\lesssim~4/3.
  % \end{align*}

  The same reasoning can be applied when $p<\frac12$ and for any
  sequence $\sigma$. The main step is to get \eqref{dif_inq} with $1$
  replaced by $\sigma$, the end of the proof being the same. In order
  to obtain this inequality, one harnesses a generalization of Russo's
  formula; we refer to \cite[Theorem~26]{Nol08} for a complete
  exposition.
\end{proof}

\subsection{Near-critical exponents}\label{sec:scaling_relations}

It is now time to relate arm-exponents to near-critical ones. The goal
of this section is to prove the following:

\begin{theorem}[Kesten \cite{Kes87}]\label{CD}
  Let $p_0\in(0,1)$ and $\ep>0$ small enough. There exist
  $c,C\in(0,\infty)$ such that for every $p\in(\frac12,p_0)$,
  \begin{align*}&c~\le~(p-1/2)L_p(\ep)^2\P\big[A_{1010}(L_p(\ep))\big]~\le~C, \\
    &c\P\big[A_1(L_p(\ep))\big]~\le~\theta(p)~\le~C
    \P\big[A_1(L_p(\ep))\big].\end{align*}
\end{theorem}

Note that we reached our original goal since Theorem~\ref{main_theorem} follows readily from Theorems~\ref{exponents} and~\ref{CD}. Indeed,
  Theorem~\ref{exponents} gives that $\P[A_{1010}(n)]=n^{-5/4+o(1)}$
  and $\P[A_1(n)]=n^{-5/48+o(1)}$.  Theorem~\ref{CD} implies that
  $\theta(p)=(p-1/2)^{5/36+o(1)}$, which is exactly the claim of
  Theorem~\ref{main_theorem}.

More generally, if we only assume the existence of $\alpha_1$ and
$\alpha_4$ such that $\P[A_{1010}(n)]=n^{-\alpha_4+o(1)}$ and
$\P[A_1(n)]=n^{-\alpha_1+o(1)}$, the previous statement implies the
existence of $\nu$ and $\beta$ such that
$L_p(\ep)=(p-1/2)^{-\nu+o(1)}$ and
$\theta(p)=(p-1/2)^{\beta+o(1)}$. Furthermore, $(2-\alpha_4) \nu = 1$
and $\beta = \alpha_1 \nu$.  This connection between different critical
exponents is called a {\em scaling relation}.

\begin{proof}[Proof of Theorem~\ref{CD}]
  In the proof, $C_1,C_2,\dots$ are constants in $(0,\infty)$
  depending on $p_0$ and $\ep$ only.  Let us deal with the second
  displayed equation first. Let $L_p=L_p(\ep)$.  On the one hand, it
  is straightforward that $\theta(p)=\Pp(0\leftrightarrow
  \infty)\le\mathbb P_{p}[A_1(L_p)]\le C_1\P[A_1(L_p)]$ thanks to
  Theorem~\ref{no_variation}.

  Since a circuit surrounding $\Lambda_{L_p}$ has length at least
  $L_p=L_{1-p}$, Proposition~\ref{exponential_decay} implies that
  $\Pp[A_1(L_p,n)]\ge C_2$ for any $n\ge L_p$.  Quasi-multiplicativity
  and Theorem~\ref{no_variation} imply
  $$\Pp[A_1(n)]\ge \Pp[A_1(L_p)]\Pp[A_1(L_p,n)]\ge C_3\Pp[A_1(L_p)]\ge C_4\P[A_1(L_p)].$$
  The claim follows by letting $n$ go to infinity.

  \medbreak

  We now turn to the first displayed equation. The right-hand inequality is
  a fairly straightforward consequence of \eqref{abcd2} and
  Theorem~\ref{no_variation}. Indeed, set $E(L_p)$ be the event that
  $[0,L_p]^2$ is crossed horizontally by an open path. Since
  $L_{p'}\ge L_p$ for $\frac12<p'\le p$, we find that
  \begin{align*}1&\ge \Pp[E(L_p)]-\P[E(L_p)]=\int_{\frac12}^p\frac{d}{dp'}\mathbb P_{p'}[E(L_p)]dp'\\
    &\ge C_5\int_{\frac12}^pL_p^2\mathbb P_{p'}[A_{1010}(L_p)]dp'\ge C_6\int_{\frac12}^pL_p^2\P[A_{1010}(L_p)]dp'\\
    &= C_7(p-\textstyle\frac12)L_p^2\P[A_{1010}(L_p)].\end{align*} The
  first equality is due to Russo's formula. The next two steps are due
  to \eqref{abcd2} followed by Theorem~\ref{no_variation}.  \medbreak
  Let us turn to the second inequality of the first displayed
  equation. Consider the torus $\mathbb T_n$ of size $n$, which can be seen as
  $\mathbb R^2$ quotiented by the following equivalence relation:
  $(x,y)\sim (x',y')$ iff $n$ divides $x-x'$ and $y-y'$. The first
  homology group of $\mathbb T_n$ is isomorphic to $\mathbb Z^2$. Let
  $[\gamma]\in \mathbb Z^2$ be the homology class of a circuit
  $\gamma$.

  Let $\mathbb T_n$ be the image of $\mathbb T$ by the canonical
  projection. A circuit of vertices in $\mathbb T$ can be identified
  to the circuit in $\mathbb T_n$ created by joining neighboring vertices by a
  segment of length $1$. Let $F(n)$ be the event that there exists a
  circuit of open vertices on $\mathbb T_n$ whose homology class in
  $\mathbb Z^2$ has non-zero first coordinate, or in other words, which is winding around $\mathbb T_n$ ``in the vertical direction''.

  If $v$ is pivotal for $F(L_p)$, there are necessarily four paths of
  alternating colors going to distance $L_p/2$ from $v$. Hence,
  $$ \frac{d}{dp'}\mathbb P_{p'}[F(L_p)]\le L_p^2\mathbb P_{p'}[A_{1010}(L_p/2)]\le C_8L_p^2\P[A_{1010}(L_p)]$$
  by quasi-multiplicativity and Theorem~\ref{no_variation}.  By
  duality, one easily obtain that $\P[F(L_p)]\le 1/2$. Now, the definition of $L_p$ together with a RSW-type argument implies that $\Pp[F(L_p)]$ is larger than $\frac34$ if $\ep$ is
  chosen small enough. Indeed, one can use a construction involving
  crossings in long rectangles. As a consequence,  \begin{equation*}\tfrac14\le \mathbb P_p(F(L_p))-\P(F(L_p))= \int_\frac12^p\frac{d}{dp'}\mathbb P_{p'}[F(L_p)]dp'\le  C_9(p-\textstyle\frac12)L_p^2\P[A_{1010}(L_p)]\end{equation*}
  hence finishing the proof.

  % A site is pivotal for $E(L_p)$ if
  % and only if there are four alternating arms starting from it and
  % going to the boundary of $[0,L_p]^2$. Except for points near the
  % boundary, this occurs with $\mathbb P_{p'}$-probability of order
  % $\mathbb P_{p'}[A_{ococ}(L_p)]$ for every
  % $p'\in(1/2,p)$. Therefore, if we neglect the effect of the
  % boundary, we obtain
  % $$1 \asymp  \Pp(A)-\P(A) \asymp \int_{1/2}^p L_p^2 \, \mathbb
  % P_{p'}[A_{ococ}(L_p)] dp'.$$ There are several ways to deal with
  % the boundary effect. One can control the probability to be pivotal
  % for boundary points separately, or one can do the following. For
  % the lower bound, it is sufficient to count points far from the
  % boundary like we did in the previous section. For the upper bound,
  % one can work with the event that the torus of size $n$ contains a
  % circuit with non-trivial homotopy. There, the probability to be
  % pivotal is the same for every site, and is smaller than
  % $\Pp[A_{ococ}(n)]$.
  %
  % Theorem~\ref{no_variation} implies $\mathbb
  % P_{p'}[A_{ococ}(L_p)]\asymp \P[A_{ococ}(L_p)]$ for every
  % $p'\in(1/2,p)$, so that
  % $$1~\asymp~(p-1/2)~L_p^2~\P[A_{ococ}(L_p)].$$
  % \medbreak We now turn to the second relation.
\end{proof}

\section{A few open questions}

\paragraph{Percolation on the triangular lattice}

Site percolation on the triangular lattice is now very well
understood, yet several questions remain open. We select
three of them.

We know the behavior of most thermodynamical quantities (the cluster
density $\theta$, the truncated mean-cluster size
$\chi(p)=(p-1/2)^{-\gamma+o(1)}$ as $p\rightarrow p_c$, the
two-point functions $\P(0\leftrightarrow x)=|x|^{-\eta+o(1)}$ as
$x\rightarrow \infty$ and many others). Nevertheless, the behavior of
the following fundamental quantity remains unproved:

\begin{question}
  Prove that the mean number of clusters per site $\kappa(p)=\mathbb
  E_p(|C|^{-1})$ behaves like $|1/2-p|^{2+\alpha+o(1)}$, where $C$ is
  the cluster at the origin and $\alpha=-2/3$.
\end{question}

Interestingly, the critical exponent for $j\neq 1$ disjoint arms of
the same color is not equal to the polychromatic arms exponent
\cite{BN10}.  A natural open question is to compute these exponents:

\begin{question}
  Compute the monochromatic exponents.
\end{question}

Even the existence of the exponents in the discrete model is not
completely understood, because we miss estimates up to constants:

\begin{question}
  Refine the error term in the arm probabilities from $(n/N)^{\alpha_j +
  o(1)}$ to $(n/N)^{\alpha_j}\Theta(1)$.
\end{question}

A result in this direction was obtained in \cite{MNW12} for a half-plane arm-event as a byproduct of a quantitative Cardy's formula (see also \cite{BCS12} for another quantitative version of Cardy's formula).
\paragraph{Percolation on other graphs}

Conformal invariance has been proved only for site percolation on the
triangular lattice. In physics, it is conjectured that the scaling
limit of percolation should be universal, meaning that it should not
depend on the lattice. For instance, interfaces of bond-percolation on
the square lattice at criticality (when the bond-parameter is 1/2)
should also converge to $\SLE(6)$.

\begin{question}
  Prove conformal invariance for critical percolation on another
  planar lattice.
\end{question}

Some progress has been made in \cite{BCL10}. For general graphs, the
question of embedding the graph becomes crucial. Indeed, if one embeds
the square lattice by gluing long rectangles, then the model will not
be rotationally invariant. We refer to \cite{Bef08a} for further
details on the subject.

\begin{question}
  For a general lattice, how may one construct a natural embedding on
  which percolation is conformally invariant in the scaling limit?
\end{question}

In order to understand universality, a natural class of lattices
consists in those for which box crossings probabilities can be
studied. Note that proofs of crossing estimates (Corollary~\ref{any_shape}) often invoke some symmetry
(rotational invariance for instance) as well as strict planarity, but
neither of these seem to be absolutely needed. A proof valid for
lattices without one of these properties would be of great
significance:

\begin{question}
  Prove crossing estimates for critical percolation on all planar (and
  possibly quasi-isometric to planar) lattices.
\end{question}

Let us mention that an important step towards the case of general
lattices was accomplished in \cite{GM11a,GM11b,GM13}, where critical
anisotropic percolation models on the hexagonal, triangular and square
lattices is studied.

Percolation in high dimension is well understood (see \emph{e.g.}
\cite{HS94} and references therein), thanks to the so-called triangle
condition and the associated lace-expansion techniques. In particular,
several critical exponents have been derived (including recently the
arm exponents \cite{KN09}) and $\theta(p_c)$ has been proved to be
equal to $0$. In intermediate dimensions, the critical phase is not
understood. For instance, one of the main conjectures in probability
is to prove that $\theta(p_c)=0$ for bond percolation on $\Z^3$. Even
weakening of this conjecture seems to be very hard. For instance, the
same question on the graph $\Z^2\times\{0,\dots,k\}$ has only been
solved very recently (see~\cite{DNS12} for site percolation in the
case $k=1$, and \cite{sandwich} for the general case).

\paragraph{Other two-dimensional models of statistical physics}

Conformal invariance (for instance of crossing probabilities) is not
restricted to percolation (see \cite{Smi06,Smi10} and references
therein). It should hold for a wide class of two-dimensional lattice
models at criticality. Among natural generalizations of percolation,
we mention the class of random-cluster models and of loop
$O(n)$-models (including the Ising model and the self-avoiding
walk). The only three models in this family for which conformal
invariance has been proved are the Ising model (the $O(n)$-model with
$n=1$), the $q\!=\!2$ random cluster model (which is a geometric
representation of the Ising model), and the uniform spanning tree.

\begin{question}
  Prove conformal invariance of another two-dimensional critical
  lattice model of percolation type.
\end{question}

\paragraph{Acknowledgments} The authors were supported by the ANR grants
BLAN06-3-134462 and MAC2 10-BLAN-0123, the EU Marie-Curie RTN CODY, the ERC AG
CONFRA, as well as by the Swiss {FNS}.

\bibliographystyle{amsalpha}
\bibliography{0_bibli_these}

\end{document}